\documentclass[oneside,11pt]{amsart}
 
\usepackage{amsmath}

\usepackage{amsfonts}
\usepackage{amssymb}
\usepackage{amsxtra}

\usepackage{amsmath}
\usepackage{amssymb}
\usepackage{graphicx, overpic}
\usepackage{caption}
\usepackage{subcaption}
\usepackage{pinlabel}
\usepackage{floatrow}
\usepackage{soul}
\usepackage{tikz}
\usetikzlibrary{decorations.markings}
\usepackage{color}
\usetikzlibrary{shapes.gates.logic.US,trees,positioning,arrows}
 \usepackage{float}

\newtheorem{thm}{Theorem}[section]
\newtheorem{prop}[thm]{Proposition}
\newtheorem{lem}[thm]{Lemma}
\newtheorem{cor}[thm]{Corollary}

\newtheorem*{main}{Main Theorem}

\theoremstyle{definition}
\newtheorem{defn}[thm]{Definition}
\newtheorem{rem}[thm]{Remark}
\newtheorem{conv}[thm]{Convention}

\newtheorem*{cor_dense}{Corollary 2.10}
\newtheorem*{cor_thick}{Corollary 2.11}
\newtheorem*{cor_geoddiv}{Corollary 2.12}
\newtheorem*{cor_mnotc}{Corollary 2.13}

\renewcommand{\emptyset}{\varnothing}

\newcommand{\field}[1]{\mathbb{#1}}

\renewcommand{\implies}{\Rightarrow}

\DeclareMathOperator{\diam}{diam}

\newcommand{\showcomments}{yes}

\usepackage{ifthen}
\newsavebox{\commentbox}
{\ifthenelse{\equal{\showcomments}{yes}}%
{\footnotemark
    \begin{lrbox}{\commentbox}
    \begin{minipage}[t]{1.25in}\raggedright\sffamily\tiny
    \footnotemark[\arabic{footnote}]}
{\begin{lrbox}{\commentbox}}}%
{\ifthenelse{\equal{\showcomments}{yes}}%
{\end{minipage}\end{lrbox}\marginpar{\usebox{\commentbox}}}
{\end{lrbox}}}

\begin{document}

\title[Divergence of finitely presented groups]{Divergence of finitely presented groups}

\begin{abstract}
We construct families of finitely presented groups exhibiting new divergence behavior; we obtain divergence functions of the form $r^\alpha$ for a dense set of exponents $\alpha \in [2,\infty)$ and $r^n\log(r)$ for integers $n \geq 2$. The same construction also yields examples of finitely presented groups which contain Morse elements that are not contracting. 
\end{abstract}

\author{Noel Brady}
\address{University of Oklahoma, Norman, OK 73019-3103, USA}
\email{nbrady@ou.edu}

\author{Hung Cong Tran}
\address{University of Oklahoma, Norman, OK 73019-3103, USA}
\email{Hung.C.Tran-1@ou.edu}

\maketitle

\section{Introduction}
In classical geometry, one can distinguish between spherical, euclidean, and hyperbolic spaces by measuring the rates at which geodesics diverge apart. 
In \cite{MR1254309}, Gersten defined \textit{divergence} as a quasi-isometry invariant of geodesic metric spaces and of finitely generated groups. Roughly speaking, divergence measures the distance distortion of the complement of an open ball of radius $r$. Gersten used divergence to classify certain $3$--manifold groups up to
quasi-isometry (see \cite{MR1302334}). The concept of divergence has also been studied by Macura \cite{MR1909513,MR3032700}, Behrstock \cite{MR2255505}, Duchin-Rafi \cite{MR2563768}, Ol'shanskii-Osin-Sapir \cite{OOS_Lacunary_hyperbolic_groups}, Dru{\c{t}}u-Mozes-Sapir \cite{MR2584607}, Behrstock-Charney \cite{MR2874959}, Behrstock-Dru{\c{t}}u \cite{MR3421592}, Sisto \cite{Sisto}, Levcovitz \cite{IL}, Gruber-Sisto \cite{MR3897973}, and others. 

It is known that the divergence of finitely presented groups can be polynomial of arbitrary degree or exponential (see \cite{MR3032700,MR1909513,MR3421592,Sisto,MR3314816}). There are examples of groups whose divergence function is not a polynomial or exponential function (see \cite{OOS_Lacunary_hyperbolic_groups,MR3897973}) but these groups are not finitely presented. 

As an application of our Main Theorem, we obtain families of finitely presented groups with diverse non-polynomial and non-exponential divergence functions.

\begin{cor_dense}
\textit{There exist finitely presented groups whose divergence is equivalent to $r^\alpha$ for a dense set of exponents $\alpha \in [2,\infty)$ and to $r^n\log(r)$ for integers $n\geq 2$.}
\end{cor_dense}
 
There are a number of geometric notions which are closely related to divergence. These include the concepts of (geometric) thickness, geodesic divergence, Morse elements, and contracting elements.  Our Main Theorem has consequences for all of these notions, which we describe briefly below. 

 In \cite{MR2501302} there is a definition of a metric space being {\em thick} of order $n$, and the corresponding definition of a group being {\em (geometrically) thick} of order $n$.  If a group is geometrically thick of order $n$, then its divergence is bounded above by $r^{n+1}$. The following corollary provides a negative answer to  Question~1.2 of \cite{MR3421592} which asked if every thick group of order $n$ has divergence equal to $r^{n+1}$.

\begin{cor_thick}
For each integer $n\geq 2$ there exists infinitely many finitely presented groups each of which is thick of order $n$ and whose divergence lies strictly between $r^n$ and $r^{n+1}$. 
\end{cor_thick}

One can focus on the divergence of a particular bi-infinite geodesic in a metric space. This leads to the notion of  \emph{geodesic divergence} (and the upper and lower geodesic divergence variants). Lower geodesic divergence was used in \cite{MR3339446,ACGH} to characterize the important concepts of Morse geodesics and Morse elements in groups. We note that when the geodesic is periodic, the upper and lower versions of geodesic divergence agree.

In \cite{MR3451473}, the second author constructed a collection of Morse geodesics whose divergence is equivalent to $r^s$ for arbitrary $s\in [2,\infty)$. However, the geodesics in \cite{MR3451473} are not periodic. As a second application of the Main Theorem, we construct a collection of periodic geodesics with diverse non-polynomial and non-exponential divergence functions. 

\begin{cor_geoddiv}
\textit{There exist finitely presented groups containing an infinite periodic geodesic whose divergence is equivalent to $r^\alpha$ for a dense set of exponents $\alpha \in [1,\infty)$ and to $r^n\log(r)$ for integers $n\geq 1$.}
\end{cor_geoddiv}


In \cite{ABD}, Abbott-Behrstock-Durham  defined the concepts of \emph{contracting quasi-geodesics} and \textit{contracting elements} in a group. Contracting elements are known to be Morse, and all previously known examples of Morse elements in finitely presented groups are contracting. Another application of our Main Theorem is the fact that Morse elements in finitely presented groups need not be contracting. 

\begin{cor_mnotc}
\textit{There exist finitely presented groups containing Morse elements which are not contracting.}
\end{cor_mnotc}


\subsection*{Acknowledgements} We thank the referee for their careful reading and very helpful comments. 

The first author was supported by Simons Foundation collaboration grant \#430097. 
The second author was supported by an AMS-Simons Travel Grant. 

\section{Definitions and statement of the main results}
\begin{conv}
\label{cv}
Let $\mathcal{M}$ be the collection of all functions from positive reals to positive reals. Let $f$ and $g$ be arbitrary elements of $\mathcal{M}$. We say that $f$ is \emph{dominated} by $g$, denoted $f \preceq g$, if there are positive constants $A, B, C$ such that $f(x) \leq g(Ax) + Bx$ for all $x > C$. We say that $f$ is \emph{equivalent} to $g$, denoted $f\sim g$, if $f\preceq q$ and $g\preceq f$.

Let $\{\delta_\rho\}$ and $\{\delta'_\rho\}$ be two families of functions of $\mathcal{M}$, indexed over $\rho \in (0,1]$. The family $\{\delta_\rho\}$ is \emph{dominated} by the family $\{\delta'_\rho\}$, denoted $\{\delta_\rho\}\preceq \{\delta'_\rho\}$, if there exists constant $L \in (0,1]$ such that $\delta_{L\rho}\preceq \delta'_\rho$ for all $\rho \in (0,1]$. We say $\{\delta_\rho\}$ is \emph{equivalent} to $\{\delta'_\rho\}$, denoted $\{\delta_\rho\}\sim \{\delta'_\rho\}$, if $\{\delta_\rho\}\preceq \{\delta'_\rho\}$ and $\{\delta'_\rho\}\preceq \{\delta_\rho\}$. If $f$ is a function in $\mathcal{M}$, the family $\{\delta_\rho\}$ is \emph{equivalent} to $f$ if there is $b\in (0,1]$ such that $\delta_\rho$ is equivalent to $f$ for each $\rho \in (0,b]$.

\end{conv}

We now recall Gersten's definition of divergence from \cite{MR1254309}. Let $X$ be a geodesic space and $x_0$ one point in $X$. Let $d_{r,x_0}$ be the induced length metric on the complement of the open ball with radius $r$ about $x_0$. If the point $x_0$ is clear from context, we will use the notation $d_r$ instead of $d_{r,x_0}$.

\begin{defn}[Group divergence]
Let $X$ be a geodesic space with a fixed point $x_0$. For each $\rho\in (0,1]$ we define a function $\delta_\rho\!: (0, \infty) \to (0, \infty)$ as follows. For each $r>0$, let $\delta_\rho(r) = \sup d_{\rho r}(x_1, x_2)$ where the supremum is taken over all $x_1$ and $x_2$ on the sphere $S(x_0,r)$ such that $d_{\rho r}(x_1,x_2)<\infty$. The family of functions $\{\delta_\rho\}$ is the \emph{divergence} of $X$.

Using Convention~\ref{cv} the divergence of $X$ does not depend on the choice of $x_0$ and it is a quasi-isometry invariant (see \cite{MR1254309}). The \emph{divergence} of a finitely generated group $G$, denoted ${\rm Div}_{G}$, is the divergence of the Cayley graph $\Gamma(G,S)$ for some (any) finite generating set $S$. 
\end{defn}

We now recall the definition of quasi-geodesic divergence and undistorted cyclic subgroup divergence.

\begin{defn}[Quasi-geodesic divergence]
Let $X$ be a geodesic space and $\alpha\!:(-\infty,\infty)\to X$ be an $(L,C)$--bi-infinite quasi-geodesic. The \emph{divergence} of $\alpha$ in $X$ is the function ${\rm Div}_\alpha\!:(0,\infty) \to (0,\infty)$ defined as follows. Let $r>0$ be an arbitrary number. If there is no path outside the open ball $B(\alpha(0),r/L-C)$ connecting $\alpha(-r)$ and $\alpha(r)$, we define ${\rm Div}_\alpha(r)=\infty$. Otherwise, we define ${\rm Div}_\alpha(r)$ to be the infimum over the lengths of all paths outside the open ball $B(\alpha(0),r/L-C)$ connecting $\alpha(-r)$ and $\alpha(r)$. 
\end{defn}

\begin{defn}[Cyclic subgroup divergence]
Let $G$ be a finitely generated group and $\langle c \rangle $ be an undistorted, infinite cyclic subgroup of $G$. Let $S$ be a finite generating set of $G$ that contains $c$. Since $\langle c\rangle$ is undistorted, every bi-infinite path with edges labeled by $c$ is a quasi-geodesic in the Cayley graph $\Gamma(G,S)$. 
The \emph{divergence} of the cyclic subgroup $\langle c \rangle$ in $G$, denoted ${\rm Div}_{\langle c \rangle}^{G}$, is defined to be the divergence of such a bi-infinite quasi-geodesic.
\end{defn}

Using Convention~\ref{cv} the divergence of the cyclic subgroup $\langle c \rangle$ in $G$ does not depend on the choice of finite generating set $S$. We leave the proof of this fact as an exercise for the reader. 

We now review the concept of Morse (contracting) quasi-geodesics and Morse (contracting) elements in a finitely generated group.

\begin{defn} [Morse quasi-geodesic]
Let $X$ be a geodesic space. A bi-infinite quasi-geodesic $\alpha$ in $X$ is \emph{Morse} in $X$ if for every $K \geq 1,C \geq 0$ there is some $M = M(K,C)$ such that every $(K,C)$--quasi–geodesic with endpoints on $\alpha$ is contained in the $M$--neighborhood of $\alpha$. 
\end{defn}

\begin{defn} [Contracting quasi-geodesic]
Let $X$ be a geodesic space. A bi-infinite quasi-geodesic $\alpha$ in $X$ is \emph{contracting} in $X$ if there exist a map $\pi_\alpha\!:X\to \alpha$ and constants $0< A< 1$ and $D\geq 1$ satisfying:
\begin{enumerate}
\item $\pi_\alpha$ is $(D,D)$--coarsely Lipschitz (i.e., $\forall x_1, x_2 \in X, d(\pi_\alpha(x_1),\pi_\alpha(x_2)) \leq Dd\bigl(x_1, x_2\bigr) + D$). 
\item For any $y\in \alpha$, $d\bigl(y,\pi_\alpha(y)\bigr)\leq D$.
\item For all $x\in X$, if we set $R=A d(x,\alpha)$, then $\diam\bigl(\pi_\alpha\bigl(B_R(x)\bigr)\bigr)\leq D$.
\end{enumerate}
\end{defn}

\begin{defn}[Morse element and Contracting element]
Let $G$ be a finitely generated group and $g$ an infinite order element in $G$. Let $S$ be a finite generating set of $G$ that contains $g$. Let $\alpha$ be a path in the Cayley graph $\Gamma(G,S)$ with edges labeled by $g$. The element $g$ in $G$ is \emph{Morse} (resp. \emph{contracting}) if the path $\alpha$ is a Morse (resp. contracting) quasi-geodesic.
\end{defn}

It is not hard to see that the notions of Morse quasi-geodesic and contracting quasi-geodesic are quasi-isometry invariants. In particular, the concepts of Morse elements and contracting elements in a finitely generated group do not depend on the choice of finite generating set $S$. 

We now define the concept of extrinsic distance function of an infinite cyclic subgroup. This concept is a key ingredient in the Main Theorem.

\begin{defn}[Extrinsic distance function]
Let $G$ be a finitely generated group with a finite generating set $S$ and let $g \in G$ be an element of infinite order. The \emph{extrinsic distance function} 
of the infinite cyclic subgroup $\langle g \rangle$
is defined as follows 
$$
\delta_{g,S}(n)\; =\; |g^n|_S.$$
We can view this as a function $(0,\infty) \to (0,\infty)$ by precomposing with the greatest integer function 
$\lfloor x\rfloor$. 
\end{defn}

\begin{defn}[Coarse Lipschitz equivalence of functions]
We say the function $\delta_{g,S}(n)$ is \emph{coarse Lipschitz equivalent} to a function $f\!:(0,\infty)\to (0,\infty)$ if there is a positive constant $K$ such that for each positive integer $n$ we have $$\frac{1}{K}f(n)-K\leq \delta_{g,S}(n)\leq K f(n)+K.$$

It is straightforward to verify that if $\delta_{g,S}(n)$ is coarse Lipschitz equivalent to a function $f\!:(0,\infty)\to (0,\infty)$, then $\delta_{g,T}(n)$ is also coarse Lipschitz equivalent to $f$ for another finite generating set $T$ of $G$.
\end{defn}

We now state the Main Theorem in this paper.

\begin{main}\label{main_thm}
Let $H$ be a finitely presented group with finite generating set $T$ containing a proper infinite cyclic subgroup $\langle c\rangle$ whose extrinsic distance function 
$$
n \mapsto |c^n|_T
$$
is coarse Lipschitz equivalent to a non-decreasing function $f$.

For each integer $m \geq 1$ there are finitely presented one-ended groups
$$
G_1 \; =\; (H\ast_{\langle c=a_0a_1^{-1}\rangle}{\mathbb Z}^2) \times {\mathbb Z}
$$
where ${\mathbb Z}^2 = \langle a_0, a_1 | a_0a_1=a_1a_0\rangle$
and 
$$
G_m \; =\; \langle G_{m-1}, a_m | a_m^{-1}a_0a_m=a_{m-1}\rangle 
$$
for $m\geq 2$, with the following properties. The subgroup $\langle a_m \rangle$ is undistorted and the following table holds.
\begin{center}
\begin{tabular}{|p{0.3in}|p{1.4in}|p{1.4in}|}
\hline
\hfill \rule[-0.3cm]{0cm}{0.8cm} {\bf $m$} & \hspace{0.3in} {\bf ${\rm Div}_{\langle a_m\rangle}^{G_m}$ } & \hspace{0.3in} {\bf ${\rm Div}_{G_m}$}\\
\hline
\hline
\rule[-0.3cm]{0cm}{0.8cm} \hfill $2$ & \hspace{0.3in}$rf(r)$ \ & \hspace{0.3in}$r^2$\\
\hline
\rule[-0.3cm]{0cm}{0.8cm} \hfill $\geq 3$ & \hspace{0.3in}$r^{m-1}f(r)$ &\hspace{0.3in}$r^{m-1}f(r)$\\
\hline
\end{tabular}
\end{center}

\end{main}

The proof of the divergence of each group $G_m$ for $m\geq 3$ is obtained from Proposition~\ref{upperkey} and Corollary~\ref{lowerkey1} while the proof of the divergence of the group $G_2$ is obtained from Proposition~\ref{upperkey2} and Corollary~\ref{lowerkey2}. We also prove each group $G_m$ is one-ended in Proposition~\ref{one-ended}. In order to talk about ${\rm Div}^{G_m}_{\langle a_m\rangle}$ we need to know that each cyclic subgroup $\langle a_m \rangle$ in $G_m$ is undistorted; this is proved by Lemma~\ref{lcool}. More precisely, for each $m\geq 2$ we choose a finite generating set $S_m$ of the group $G_m$ such that the bi-infinite paths $\alpha_m$ in the Cayley graph $\Gamma(G_m,S_m)$ with edges labeled by $a_m$ are geodesic. Then we prove the divergence of $\alpha_m$ is equivalent to $r^{m-1}f(r)$ in Proposition~\ref{cyclicdivergence}. Thus, divergence of cyclic subgroup $\langle a_m \rangle$ in $G_m$ is equivalent to $r^{m-1}f(r)$. 

We note that although the divergence of the bi-infinite geodesic $\alpha_2$ in the Cayley graph $\Gamma(G_2,S_2)$ is equivalent to $rf(r)$, this is not the divergence of the ambient group $G_2$. In fact, we can always find from the construction of $G_2$ another bi-infinite geodesic in $\Gamma(G_2,S_2)$ with quadratic divergence (see Proposition~\ref{pcoolcool}) and Proposition~\ref{upperkey2} also shows that the divergence of $G_2$ is dominated above by a quadratic function.   

The corollaries below provide four important applications of the Main Theorem. 

\begin{cor}\label{cor_dense}
There exist finitely presented groups whose divergence is equivalent to $r^\alpha$ for a dense set of exponents $\alpha \in [2,\infty)$ and to $r^n\log(r)$ for integers $n\geq 2$.
\end{cor}

\begin{proof}
First we establish the $r^\alpha$ divergence functions. Given integers $p > q \geq 2$, let $H$ be a snowflake group with defining Perron-Frobenius matrix the 1-by-1 matrix $(q)$ and scaling factor $p$ and let $c$ be a generator of an edge group in the snowflake construction of \cite{BBFS}. Then by \cite{BBFS}, Proposition 4.5, the extrinsic length function of the $\langle c\rangle$ subgroup of $H$ is equivalent to $f(r) = r^{\log_p(q)}$. The Main Theorem implies that for $m \geq 3$, the finitely presented group $G_m$ has divergence equivalent to $r^{m-1+\log_p(q)}$. Taking $\alpha = m-1 + \log_p(q)$ gives the first result. 

For the functions $r^n\log(r)$, choose $H$ in the proof of the Main Theorem to be the Baumslag-Solitar group $\langle c,t \, |\, tct^{-1}=c^2\rangle $. 
By \cite{MR3384081} for example, 
 the extrinsic distance function of $\langle c \rangle$ in $H$ is coarse Lipschitz equivalent to the function $\log(r)$. The Main Theorem implies that the finitely presented groups $G_m$ for $m \geq 3$ have divergence equivalent to $r^{m-1}\log(r)$. Taking $n = m-1$ yields the second result. 
\end{proof}

\begin{cor}\label{cor_thick}
For each integer $n\geq 2$ there exists infinitely many finitely presented groups each of which is thick of order $n$ whose divergence lies strictly between $r^n$ and $r^{n+1}$. 
\end{cor}

\begin{proof}
For each integer $m\geq 3$ and integers $p>q\geq 2$, the groups $G_m$ defined in the first paragraph of the proof of Corollary~\ref{cor_dense} above are thick of order $m-1$ (proof given in the next paragraph) and have divergence $r^{m-1+\log_p(q)}$. Setting $n=m-1$ yields the result. 

In order to establish thickness of the $G_m$, first note that each group $G_1$ is thick of order zero (also known as unconstricted) because it is a direct product (see example 1 following Definition 3.4 of \cite{MR2501302}). By induction on $m$ and the recursive construction of the $G_m$, it follows from Definition~7.1 of \cite{MR2501302} that the groups $G_m$ are thick of order at most $m-1$. If $G_m$  were thick of order less than $m-1$, Corollary~4.17 of \cite{MR3421592} would imply that they have divergence at most $r^{m-1}$. Since the groups $G_m$ have divergence strictly greater than $r^{m-1}$, we conclude that they are thick of order exactly $m-1$. 
\end{proof}

\begin{cor}\label{cor_geodesic}
There exist finitely presented groups containing an infinite periodic geodesic whose divergence is equivalent to $r^\alpha$ for a dense set of exponents $\alpha \in [1,\infty)$ and to $r^n\log(r)$ for integers $n\geq 1$.
\end{cor}

\begin{proof}
The Main Theorem implies that for each $m\geq 2$ the divergence of an infinite periodic geodesic $\beta$ with edges labeled by $a_m$ in group $G_m$ is equivalent to the function $r^{m-1}f(r)$. Given integers $p > q \geq 2$, let $H$ be the finitely presented group and let $\langle c \rangle$ be the subgroup of $H$ with the extrinsic length function equivalent to $f(r) = r^{\log_p(q)}$ as in the proof of Corollary\ \ref{cor_dense} above. Then, the divergence of $\beta$ is equivalent to $r^{m-1+\log_p(q)}$. Taking $\alpha = m-1 + \log_p(q)$ gives the first result. 

Similarly, let $H$ be the finitely presented group and let $\langle c \rangle$ be the subgroup of $H$ with the extrinsic length function equivalent to $f(r) = \log(r)$ as in the proof of Corollary\ \ref{cor_dense} above. Then, the divergence of $\beta$ is equivalent to $r^{m-1}\log(r)$. Taking $n = m-1$ yields the second result.
\end{proof}

\begin{cor}\label{cor_m_not_c}
There exist finitely presented groups containing Morse elements which are not contracting.
\end{cor}

\begin{proof}
By Corollary~\ref{cor_geodesic} there are finitely presented groups $G_2$ and undistorted infinite cyclic subgroups $\langle a_2 \rangle$ whose divergence functions are equivalent to $r^\alpha$ for a dense set of $\alpha \in (1,2)$ or to $r\log(r)$. By Theorem 1.3 in \cite{ACGH} the group element $g$ is Morse. If $g$ is a contracting element, we can prove the divergence of the cyclic subgroup $\langle g \rangle$ is at least quadratic by using an analogous argument as in the proof of Theorem 2.14 in \cite{MR3339446}. (We provide a proof of this fact for the reader's convenience in Proposition~\ref{lem_CS}). This is a contradiction. Therefore, $g$ is not a contracting element. 
\end{proof}

\begin{rem}
\label{remrem}
 We refer the reader to the work of Ol'shanskii-Sapir~\cite{MR1829048} for more examples of extrinsic distance functions of cyclic subgroups of finitely presented groups. These will furnish new variants of Corollaries~\ref{cor_dense}, ~\ref{cor_thick},  ~\ref{cor_geodesic}, and ~\ref{cor_m_not_c} above. 
\end{rem}

\begin{rem}
\label{flag_planting}
One can ask about the behavior of higher dimensional divergence functions of groups with appropriate finiteness conditions. There are a number of interesting ways of generalizing the examples and constructions in this paper to approach analogous questions about higher dimensional divergence. 
\end{rem}

\section{Algebraic and geometric properties of the groups $G_m$}





We fix generating sets $S_m$ for the groups $G_m$ in the Main Theorem. 
\begin{defn}[Generating sets $S_m$]
The group $G_1$ is defined in the Main Theorem as an amalgam 
$$
G_1 \; =\; (H\ast_{\langle c=a_0a_1^{-1}\rangle}{\mathbb Z}^2) \times {\mathbb Z}
$$
where $H$ has finite generating set $T$ and the free abelian group ${\mathbb Z}^2$ is generated by $a_0$ and $a_1$. We may assume $T$ contains the infinite order element $c$. Let $b$ be a generator for the ${\mathbb Z}$ factor. Then $S_1=T\cup \{a_0,a_1,b\}$ is a finite generating set for $G_1$. 

For each integer $m \geq 2$ the groups $G_m$ are defined recursively by 
$$
G_m \; =\; \langle G_{m-1}, a_m \, |\, a_m^{-1}a_0a_m=a_{m-1} \rangle
$$
and so have recursively defined finite generating sets 
$S_m=S_{m-1}\cup \{a_m\}$.
\end{defn}

The following lemma will be used repeatedly to prove the group divergence by induction in the Main Theorem.

\begin{lem}\label{ll1}
Let $m\geq 2$ be an integer. Then the inclusion map $i\!:\Gamma(G_{m-1},S_{m-1})\hookrightarrow \Gamma(G_{m},S_{m})$ is an isometric embedding.
\end{lem}

\begin{proof}
The group $G_m$ is an isometric HNN extension in the sense of \cite{MR1668335}, with base group $G_{m-1}$ and edge groups $\langle a_0\rangle$ and $\langle a_{m-1}\rangle$. The homomorphism $G_{m-1} \to {\mathbb Z}$ taking all the $a_j$ ($0\leq j \leq m-1$) to a generator of ${\mathbb Z}$ and all other generators of $G_{m-1}$ to the identity shows that each $\langle a_j\rangle$ is a retract of $G_{m-1}$ and so are isometrically embedded subgroups. By Lemma 2.2(2) in \cite{MR1668335} (taking $G=H=G_{m-1}$) the inclusion $G_{m-1} \hookrightarrow G_m$ is an isometric embedding. 
\end{proof}

The following two lemmas will be used to define the concepts of $k$--corner (see Definition~\ref{corner1} and Definition~\ref{corner2}) which appear in the proof of the lower bound of group divergence in the Main Theorem.

\begin{lem}
\label{lcool}
Let $0\leq i< j \leq m$ be integers such that $j\geq 2$. Let $p,q$ be arbitrary integers. Then $|a_i^{p} a_j^{q}|_{S_m}=|p|+|q|$.
\end{lem}

\begin{proof}

There is a group homomorphism $\Psi\!: G_{j} \to \mathbb{Z}^2$ taking $a_{j}$ to the generator $(0,1)$, taking each $a_i$ to the generator $(1,0)$ for $0\leq i \leq j-1$, and taking each element in $S_{j}-\{a_0,a_1,\cdots,a_{j}\}$ to the identity $(0,0)$. Therefore, $\Psi(a_i^{p} a_{j}^{q})=(p,q)$. This implies that
$$|a_i^{p} a_{j}^{q}|_{S_{j}}\geq |\Psi(a_i^{p} a_{j}^{q})|=|p|+|q|.$$

Since $a_i$ and $a_j$ are also elements in the finite generating set $S_j$, we have $$|a_i^{p} a_{j}^{q}|_{S_{j}}=|p|+|q|.$$ 

Also, the inclusion map $i\!:\Gamma(G_{j},S_{j})\hookrightarrow \Gamma(G_{m},S_{m})$ is an isometric embedding by Lemma~\ref{ll1}. Therefore, $|a_i^{p} a_{j}^{q}|_{S_{m}}=|a_j^{p} a_{j}^{q}|_{S_{j}}=|p|+|q|$.
\end{proof}

\begin{lem}
\label{lcool2}
Let $0\leq i< j <k\leq m$ be integers. Let $p,q,n$ be arbitrary integers such that $a_i^{p}=a_j^{q}a_k^{n}$ in $G_m$. Then $p=q=n=0$.
\end{lem}

\begin{proof}
We first prove that $n=0$. In fact, there is a group homomorphism $\Phi\!: G_{k} \to \mathbb{Z}$ taking $a_{k}$ to $1$ and each generators in $S_k-\{a_k\}$ to $0$. Therefore, $\Phi(a_i^{p})=0$ and $\Phi(a_j^{q}a_k^{n})=n$ which imply that $n=0$. Thus, $a_i^{p}=a_j^{q}$ which implies that $p=q=0$ by Lemma~\ref{lcool}. 
\end{proof}

Let $f$ be a coarse Lipschitz representative for the extrinsic length function of the subgroup $\langle c\rangle$ in $H$. This means that there exists a constant $C\geq 1$ such that 
$$
\frac{1}{C}f(n) - C \; \leq \; |c^n|_T\; \leq \; Cf(n)+C 
$$
for each positive integer $n$. In order to simplify the algebra in several proofs in this paper, we will choose the coarse Lipschitz representative $f$ of $|c^n|_T$ so that it satisfies the following two conditions:
\begin{itemize}
    \item There exists a constant $C_1 \geq 1$ such that $f(n)\, \leq\,  |c^n|_T\, \leq \, C_1f(n)+C_1 $ for each positive integer $n$, and 
    \item $f(x) \, \leq \, x$ for all $x \in (0,\infty)$. 
\end{itemize}

 
 The following lemma establishes inequalities which will be used in the proofs of the upper and lower bounds on the divergence in the Main Theorem. 
 
\begin{lem}
\label{a1}
Let $H$ be as in the Main Theorem and let $c \in H$ have extrinsic length function satisfying 
$$
f(k)\; \leq \;|c^k|_T \; \leq \; C_1f(k)+C_1 
$$ 

for some function $f$. Then for all pairs of integers $n, m$ we have:
\begin{enumerate}
    \item $d_{S_1}(a_0^n,a_1^n)\leq C_1f(|n|)+C_1$. Moreover, there is a geodesic in the Cayley graph $\Gamma(G_1,S_1)$ which connects $a^n_0$ and $a^n_1$ and lies outside the open ball $B(e,|n|)$;
    \item $|a_0^ma_1^n|_{S_1}\geq f(|n|)$.
\end{enumerate}
\end{lem}

\begin{proof}
We first claim that for each integer $n$ and each group element $h\in H$ we have $|a_0^nh|_{S_1}\geq |n|$ and $|a_1^nh|_{S_1}\geq |n|$. In fact, there is a group homomorphsim $\Phi\!: G_1 \to \field{Z}$ that maps each element in $S_1-\{a_0,a_1\}$ to $0$ and maps both $a_0$, $a_1$ to $1$. Therefore, $|\Phi(g)|\leq |g|_{S_1}$. Also, $\Phi(a_0^nh)=\Phi(a_1^nh)=n$. This implies that $|a_0^nh|_{S_1}\geq |n|$ and $|a_1^nh|_{S_1}\geq |n|$.

We first prove Statement (1). It is straightforward to see that there is a group homomorphism $\Psi\!: G_1 \to H$ that maps $a_0$ to $c$, maps $a_1$ to $e$, maps $b$ to $e$ and maps each element in $T$ to itself. Therefore, for each $g$ in $G_1$ we have $|\Psi(g)|_T\leq |g|_{S_1}$. Also, $\Psi(c^n)=c^n$. Then, $|c^n|_T\leq |c^n|_{S_1}$. On the other hand, $|c^n|_{S_1}\leq |c^n|_T$ because $T$ is a subset of $S_1$. Therefore, $|c^n|_T =|c^n|_{S_1}$. Since $c^n=a_1^{-n}a_0^n$ and $|c^n|_T\leq C_1f(|n|)+C_1$, we have $d_{S_1}(a_0^n,a_1^n)\leq C_1f(|n|)+C_1$. 

Let $\gamma_0$ be a geodesic in the Cayley graph $\Gamma(H,T)\subset \Gamma(G_1,S_1)$ connecting $e$ and $c^n$. Since $|c^n|_T =|c^n|_{S_1}$, the path $\gamma_0$ is also a geodesic in $\Gamma(G_1,S_1)$. Therefore, $\gamma=a_1^n\gamma_0$ is a geodesic in $\Gamma(G_1,S_1)$ connecting $a_1^n$ and $a_0^n$. Each vertex in $\gamma$ has the form $a_1^nh$ for some $h\in H$ and such a vertex lies outside the open ball $B(e,|n|)$ by paragraph one above. Therefore, $\gamma$ lies outside the open ball $B(e,|n|)$.

We now prove Statement (2). It is straightforward to see that there is a group homomorphism $\Psi'\!: G_1 \to H$ that maps $a_0$ to $e$, maps $a_1$ to $c^{-1}$, maps $b$ to $e$ and maps each element in $T$ to itself. Therefore, for each $g$ in $G_1$ we have $|\Psi'(g)|_T\leq |g|_{S_1}$. Also, $\Psi'(a_0^ma_1^n)=c^{-n}$. Therefore, $|a_0^ma_1^n|_{S_1}\geq |\Psi'(a_0^ma_1^n)|_T=|c^{-n}|_T\geq f(|n|)$.
\end{proof}

We now discuss some geometric properties of the group $G_1$ and these properties will be used in the proof of the Main Theorem. 
\begin{lem}
\label{l1}
Let $\{\delta_\rho\}$ be the divergence of the Cayley graph $\Gamma(G_1,S_1)$. Then there is a constant $C_2\geq 1$ such that $\delta_1(r)\leq C_2r+C_2$ for each $r\geq 0$. 
\end{lem}

\begin{proof}

The linear upper bound on divergence in this case follows from the direct product structure of $G_1$ and the fact that ${\mathbb Z}$ has extendable geodesics. We give a proof of this fact for completeness below. The proof is similar to that of Lemma~7.2 of \cite{MR3073670}. 

Assume that $G$ and $H$ are infinite, finitely generated groups and that the Cayley graph of $H$ has extendable geodesics. If $\{\delta_\rho\}$ is the divergence of the Cayley graph of the group $G\times H$, then there is a constant $C_2\geq 1$ such that $\delta_1(r)\leq C_2r+C_2$ for each $r\geq 0$. 

Given two elements $(g_1,h_1)$ and $(g_2,h_2)$ on the $n$--sphere centered at the identity in the Cayley graph of $G\times H$, we show how to connect them by a linear length path that avoids the open $n$--ball centered at the identity in the Cayley graph of $G\times H$. To this end, we extend the geodesic $[1,h_1]$ to a geodesic $[1,h_1k_1]$ with endpoint on the $n$-sphere centered at the identity in the Cayley graph of $H$. Likewise we extend the geodesic $[1,h_2]$ to a geodesic $[1,h_2k_2]$ with endpoint on the $n$-sphere centered at the identity in the Cayley graph of $H$. Pick a point $g_0$ on the $n$-sphere centered at the identity in the Cayley graph of $G$.

The following concatenation of paths avoids the $n$-ball centered at the identity in the Cayley graph of $G\times H$ and has length at most $8n$. The arrows indicate geodesic paths in the Cayley graph of $G\times H$ which are geodesic in one factor and constant paths in the other factor. 
$$
(g_1,h_1) \; \stackrel{\leq n} \longrightarrow \; (g_1,h_1k_1) 
 \; \stackrel{\leq 2n}\longrightarrow \; (g_0,h_1k_1) 
  \; \stackrel{\leq 2n}\longrightarrow \; (g_0,h_2k_2) 
   \; \stackrel{\leq 2n}\longrightarrow \; (g_2,h_2k_2) 
    \; \stackrel{\leq n}\longrightarrow \; (g_2,h_2) 
$$

\end{proof}


The following lemma is a direct result of Lemma~\ref{l1} and it will be used for the proof of the upper bound of the group divergence in the Main Theorem. 

\begin{lem}
\label{s1}
There is a constant $C_3>0$ such that for each non-zero integer $n$ we have the following holds:
\begin{enumerate}
    \item There is a path in $\Gamma(G_1,S_1)$ which connects $a_0^{-n}$ and $a_0^{n}$, lies outside the open ball $B(e,|n|)$, and has length at most $C_3|n|+C_3$.
    \item For each group element $s\in S_1\cup S_1^{-1}$ there is a path in $\Gamma(G_1,S_1)$ which connects $a_0^{2n}$ and $sa_0^{2n}$, lies outside the open ball $B(e,|n|)$, and has length at most $C_3|n|+C_3$. 
\end{enumerate}
\end{lem}

\begin{proof}
Statement (1) follows directly from Lemma~\ref{l1} with $C_3 \geq C_2$. Statement (2) also follows directly from Lemma~\ref{l1} with $C_3 = C_2 + 2$ since the paths based at $e$ and labeled by $a_0^{2n}$ and $sa_0^{2n}$ intersect $S(e,|n|)$. 
\end{proof}

\section{The divergence of $G_m$}

\subsection{The upper bound}

The following two lemmas will be used in the proof of the upper bound of the group divergence in the Main theorem (see Proposition~\ref{upperkey} and Proposition~\ref{upperkey2}). Note that Lemma~\ref{basic} is a variation of Lemma 5.5 in \cite{MS2018}. We include its proof here for the convenience of the reader.

\begin{lem}
\label{basic}
Let $X$ be a geodesic space and $x_0$ be a point in $X$. Let $r$ be a positive number and let $x$ be a point on the sphere $S(x_0,r)$. Let $\alpha_1$ and $\alpha_2$ be two rays with the same initial point $x$ such that $\alpha_1\cup\alpha_2$ is a bi-infinite geodesic. Then either $\alpha_1$ or $\alpha_2$ has an empty intersection with the open ball $B(x_0,r/2)$.
\end{lem}

\begin{proof}
Assume by the way of contradiction that both rays $\alpha_1$ and $\alpha_2$ have non-empty intersection with the open ball $B(x_0,r/2)$. Thus, there is $u\in \alpha_1$ and $v\in \alpha_2$ such that $d(x_0,u)<r/2$ and $d(x_0,v)<r/2$. Since $d(x_0,x)=r$, then by the triangle inequality we have $d(x,u)>r/2$ and $d(x,v)>r/2$. Also, $\alpha_1\cup\alpha_2$ is a bi-infinite geodesic. Therefore, we have $$d(u,v)=d(u,x)+d(x,v)>r/2+r/2>r.$$ By the triangle inequality again, we have $$d(u,v)\leq d(u,x_0)+d(x_0,v)< r/2+r/2<r$$
which is a contradiction. Therefore, at least one of $\alpha_1$ and $\alpha_2$ has an empty intersection with the open ball $B(x_,r/2)$.

Note that this can be interpreted as a type of quasi-convexity result; namely, it states that the $r/2$--ball is ``$r/2$--quasi-convex.''
\end{proof}


\begin{lem}
\label{ag}
For each integer $m\geq 2$ there are constants $M_m$ and $N_m$ such that the following two statements hold:
\begin{itemize}
    \item[$(P_m)$] For each non-zero integers $n$ and $r$ such that $|n|=r$ and $\epsilon \in \{-1,+1\}$ there is a path in $\Gamma(G_m,S_m)$ which connects $a_0^{2n}$ and $a_m^{\epsilon}a_0^{2n}$, lies outside the open ball $B(e,r)$, and has length at most $M_mr^{m-2} \bigl(f(M_mr)+1\bigr)$.
    \item[$(Q_m)$] For each non-zero integers $n_1$, $n_2$ and $r$ such that $|n_1|=|n_2|=r$ there is a path in $\Gamma(G_m,S_m)$ which connects $a_0^{n_1}$ and $a_m^{n_2}$, lies outside the open ball $B(e,r)$, and has length at most $N_mr^{m-1} \bigl(f(N_mr)+1\bigr)$.
\end{itemize}
\end{lem}

\begin{proof}
We prove the above lemma by induction on $m$ using the following strategy. We first prove Statement $P_2$. Then we prove the implication $(P_m)\implies (Q_m)$ for each $m\geq 2$. Finally, we prove the implication $(Q_m)\implies (P_{m+1})$ for each $m\geq 2$. 

For Statement $(P_2)$ we only prove it for the case of $\epsilon=-1$ and the proof for the case of $\epsilon=1$ is almost identical. Let $C_1$ be the constant in Lemma~\ref{a1}. By Lemma~\ref{ll1} and Lemma~\ref{a1} there is a path $\alpha$ in $\Gamma(G_2,S_2)$ which connects $a_0^{2n}$ and $a_1^{2n}$, lies outside the open ball $B(e,r)$, and has length at most $C_1 f(2r)+C_1$. Since $a_2^{\epsilon}a_0^{2n}=a_2^{-1}a_0^{2n}=a_1^{2n}a_2^{-1}$, we can connect $a_1^{2n}$ and $a_2^{\epsilon}a_0^{2n}$ by an edge $f$ labeled by $a_2$ and this edge lies outside the open ball $B(e,r)$. Therefore, $\alpha\cup f$ is a path in $\Gamma(G_2,S_2)$ which connects $a_0^{2n}$ and $a_2^{\epsilon}a_0^{2n}$, lies outside the open ball $B(e,r)$, and has length at most $C_1 f(2r)+C_1+1$. Therefore, Statement $P_2$ is proved by choosing an appropriate constant $M_2 \geq \max\{C_1 + 1, 2\}$.

We now prove the implication $(P_m)\implies (Q_m)$ for each $m\geq 2$. We only prove Statement $(Q_m)$ for the case $n_2>0$ and the proof for the case $n_2<0$ is almost identical. We first connect $a_0^{n_1}$ and $a_0^{4n_1}$ by a path $\eta_0$ labeled by the word $a_0^{3n_1}$. By Lemma~\ref{lcool}, the path $\alpha_1$ lies outside the open ball $B(e,r)$ and has the length exactly $3r$. By Statement $(P_m)$ there is a path $\eta_1$ in $\Gamma(G_m,S_m)$ which connects $a_0^{4n_1}$ and $a_ma_0^{4n_1}$, lies outside the open ball $B(e,2r)$, and has length at most $M_m(2r)^{m-2} \bigl(f(2M_mr)+1\bigr)$. 

For each $1\leq i \leq n_2$ let $\eta_i=a_m^{i-1}\eta_1$. Then each $\eta_i$ is a path in the Cayley graph $\Gamma(G_m,S_m)$ which connects $a_m^{i-1}a_0^{4n_1}$ and $a_m^{i}a_{0}^{4n_1}$, lies outside the open ball $B(a_m^{i-1},2r)$, and has length at most $M_m(2r)^{m-2} \bigl(f(2M_mr)+1\bigr)$. Since $d_{S_{m}}(e,a_m^{i-1})= i-1<n_2=r$, each path $\eta_i$ lies outside the open ball $B(e,r)$. Let $\eta_{n_2+1}$ be the path connecting $a_m^{n_2}a_{0}^{4n_1}$ and $a_m^{n_2}$ labeled by $a_{0}^{-4n_1}$. Then $\eta_{n_2+1}$ has length exactly $4r$ and each vertex in $\eta_{n_2+1}$ has the form $a_m^{n_2}a_0^j$ for $0\leq j\leq 4n_1$. By Lemma~\ref{lcool} we have $|a_m^{n_2}a_0^j|_{S_m}=n_2+j\geq r$. Therefore, $\eta_{n_2+1}$ lies outside the open ball $B(e,r)$. Let $\eta=\eta_0\cup\eta_1\cup\eta_2\cup \cdots\cup\eta_{n_2}\cup\eta_{n_2+1}$. Then $\eta$ is a path in $\Gamma(G_m,S_m)$ which connects $a_0^{n_1}$ and $a_m^{n_2}$, lies outside the open ball $B(e,r)$, and has length at most $r\bigg[M_m(2r)^{m-2} \bigl(f(2M_mr)+1\bigr)\bigg]+7r$. Therefore, Statement $(Q_m)$ is proved by choosing an appropriate constant $N_m \geq \max\{2^{m-2}M_m, 2M_m\}+7$.

We now prove the implication $(Q_m)\implies (P_{m+1})$ for each $m\geq 2$. For Statement $(P_{m+1})$ we only prove it for the case of $\epsilon=-1$ and the proof for the case of $\epsilon=1$ is almost identical. By Lemma~\ref{ll1} and Statement $(Q_m)$ there is a path $\alpha$ in $\Gamma(G_{m+1},S_{m+1})$ which connects $a_0^{2n}$ and $a_m^{2n}$, lies outside the open ball $B(e,r)$, and has length at most $N_m(2r)^{m-1} \bigl(f(2N_mr)+1\bigr)$. Since $a_{m+1}^{\epsilon}a_0^{2n}=a_{m+1}^{-1}a_0^{2n}=a_m^{2n}a_{m+1}^{-1}$, we can connect $a_m^{2n}$ and $a_{m+1}^{\epsilon}a_0^{2n}$ by an edge $f$ labeled by $a_{m+1}$ which lies outside the open ball $B(e,r)$. Therefore, $\alpha\cup f$ is a path in $\Gamma(G_{m+1},S_{m+1})$ which connects $a_0^{2n}$ and $a_{m+1}^{\epsilon}a_0^{2n}$, lies outside the open ball $B(e,r)$, and has length at most $N_m(2r)^{m-1} \bigl(f(2N_mr)+1\bigr)+1$. Therefore, Statement $(P_{m+1})$ is proved by choosing an appropriate constant $M_{m+1} \geq 2^{m-1}N_m + 1$.
\end{proof}

We are now ready to prove the upper bound of the divergence of the group $G_m$ for $m\geq 3$. The proof of the next two propositions follow the same strategy. Here is the geometric intuition behind this strategy. The basic idea is to connect an arbitrary point $x$ on the sphere $S(e,r)$ to one of the points $a_0^r$ or $a_0^{-r}$ by a path lying outside the ball $B(e,r/2)$ with desired control on its overall length. This is achieved as follows. 
\begin{itemize}
    \item Consider the bi-infinite geodesic path through the point $x$ with edges labeled by $a_0$. Because this path is a geodesic, Lemma~\ref{basic} implies that the two rays emanating from $x$ cannot both penetrate $B(e,r/2)$. Suppose that the positive $a_0$-ray emanating from $x$ does not penetrate $B(e,r/2)$.
    \item Construct a ``comb" by attaching a path (``tooth") labeled $a_0^{4r}$ at each vertex along a geodesic from $e$ to $x$. We argue that it is possible to connect the endpoints of successive teeth of this comb by paths which avoid $B(e,r)$ and which have controlled total length. The upper bounds on these lengths come from Lemma~\ref{ag} and Lemma~\ref{s1}.  
\end{itemize}
In the proof of Proposition~\ref{upperkey} we have to consider all the stable letters $a_3, \ldots, a_m$. The proof of Proposition~\ref{upperkey2} is easier since we only are dealing with a single generator $a_2$ which lies outside of a group with linear divergence.

\begin{prop}
\label{upperkey}
Let $m\geq 3$ be an integer. Let $\{\delta_\rho\}$ be the divergence of the Cayley graph $\Gamma(G_m,S_m)$. Then there is a constant $A_m$ such that for each $\rho \in (0,1/2]$ we have $$\delta_\rho(r)\leq A_m r^{m-1}\bigl(f(A_mr)+1\bigr) \text{ for each } r\geq 1.$$ In particular, $\delta_\rho(r) \preceq r^{m-1}f(r)$ for each $\rho \in (0,1/2]$,
\end{prop}

\begin{proof}
We can assume that $r$ is an integer. Let $C_3$ be the constant in Lemma~\ref{s1}. Then there is a path which connects $a_0^{-r}$ and $a_0^{r}$, lies outside the open ball $B(e,r)$, and has length at most $C_3r+C_3$. It suffices to show there is a constant $B_m$ depending only on $m$ such that for each point $x$ on the sphere $S(e,r)$ we can either connect $x$ to $a_0^r$ or connect $x$ to $a_0^{-r}$ by a path outside the open ball $B(e,r/2)$ with the length at most $B_m r^{m-1}\bigl(f(B_mr)+1\bigr)$. Now the proposition follows choosing a suitable constant $A_m \geq 2B_m + 2C_3$. 

Now we establish the $B_m r^{m-1}\bigl(f(B_mr)+1\bigr)$ bound. Let $\alpha$ be a bi-infinite geodesic which contains $x$ and has edges labeled by $a_0$. Then $\alpha$ is the union of two rays $\alpha_1$ and $\alpha_2$ that share the initial point $x$. Assume that $\alpha_1$ traces each edge of $\alpha$ in the positive direction and $\alpha_2$ traces each edge of $\alpha$ in the negative direction. By Lemma~\ref{basic}, either $\alpha_1$ or $\alpha_2$ (say $\alpha_1$) lies outside the open ball $B(e,r/2)$. We now construct a path $\eta$ which connects $x$ and $a_0^r$, lies outside the open ball $B(e,r/2)$, and has the length at most $B_m r^{m-1}\bigl(f(B_mr)+1\bigr)$. We note that we can use an analogous argument to construct a similar path connecting $x$ and $a_0^{-r}$ in the case $\alpha_2$ lies outside the open ball $B(e,r/2)$.

\noindent
{\em Constructing the comb and connecting endpoints of successive teeth.} First, we connect $a_0^r$ and $a_0^{4r}$ by the geodesic $\eta_0$ labeled by $a_0^{3r}$. Then $\eta_0$ lies outside the open ball $B(e,r)$ and has length exactly $3r$. Since $|x|_{S_m}=r$, we can write $x=s_1s_2\cdots s_{r-1}s_r$ where $s_i\in S_m\cup S_m^{-1}$. Let $R_m$ be a constant which is greater than four times of the constant $C_3$ in Lemma~\ref{s1} and greater than $4^m$ times of each constant $M_i$ for $2\leq i \leq m$ in Lemma~\ref{ag}. By Lemma~\ref{ll1}, Lemma~\ref{s1}, and Lemma~\ref{ag} we can connect $a_0^{4r}$ and $s_1a_0^{4r}$ by a path $\eta_1$ which lies outside the open ball $B(e,2r)$ and has length at most $R_m r^{m-2}\bigl(f(R_mr)+1\bigr)$.

Similarly for $2\leq i\leq r$ we can also connect $(s_1s_2\cdots s_{i-1})a_0^{4r}$ and $(s_1s_2\cdots s_{i})a_0^{4r}$ by a path $\eta_i$ which lies outside the open ball $B(s_1s_2\cdots s_{i-1},2r)$ and has length at most $R_m r^{m-2}\bigl(f(R_mr)+1\bigr)$. Finally, we connect $xa_0^{4r}$ and $x$ using the subsegment $\eta_{r+1}$ of $\alpha_1$. Then $\eta_{r+1}$ lies outside the open ball $B(e,r/2)$ and has length exactly $4r$. 

Let $\eta=\eta_0\cup\eta_1\cup\eta_2\cup\cdots\cup \eta_r\cup \eta_{r+1}$. Then $\eta$ is a path which connects $a_0^r$ and $x$, lies outside the open ball $B(e,r/2)$, and has length at most $R_m r^{m-1}\bigl(f(R_mr)+1\bigr)+7r$. Since $f$ is a non-decreasing function, the length of $\eta$ is bounded above by $B_m r^{m-1}\bigl(f(B_mr)+1\bigr)$ by some appropriate choice of $B_m \geq R_m + 7$.
\end{proof}

The following proposition gives the quadratic upper bound for the divergence of the group $G_2$. The proof of this proposition proceeds exactly as the proof of Proposition~\ref{upperkey} and we leave it to the reader. The main difference is that we only need the upper bounds on lengths on paths avoiding $B(e, 2r)$ which connect $a_0^{4r}$ with $a_2a_0^{4r}$ from statement $(P_2)$ in Lemma~\ref{ag}. This bounds path lengths by $M_2f(M_2r) + M_2$, and since $f(r)$ is sub-linear, we can replace this bound by a linear function of $r$.  The other length estimates are on paths avoiding $B(e,2r)$ which connect $a_0^{4r}$ with $sa_0^{4r}$ for $s$ a generator of $G_1$ and these are also linear. 

\begin{prop}
\label{upperkey2}
Let $\{\delta_\rho\}$ be the divergence of the Cayley graph $\Gamma(G_2,S_2)$. Then there is a constant $A$ such that for each $\rho \in (0,1/2]$ we have $$\delta_\rho(r)\leq A r^2 +Ar \text{ for each $r$ sufficiently large}.$$ In particular, $\delta_\rho(r) \preceq r^2\text{ for each } \rho \in (0,1/2]$,
\end{prop}

We end this subsection by proving each group $G_m$ in the Main Theorem is one-ended.

\begin{prop}
\label{one-ended}
Let $m\geq 1$ be an integer. Then the group $G_m$ is one-ended.
\end{prop}

\begin{proof}


$G_1$ is one-ended because it is a direct product of two infinite groups.
For $m \geq 2$ note that (by retracting along geodesic rays based at $e$)  it is sufficient to prove that any two points on the sphere $S(e,r)$ can be connected by a path avoiding $B(e,r/2)$ for large integers $r$. In the proof of Proposition~\ref{upperkey} and Proposition~\ref{upperkey2} we construct paths which avoid $B(e,r/2)$ and which connect an arbitrary $x\in S(e,r)$ to one of $a_0^r$
or $a_0^{-r}$, and we can connect $a_0^r$ to $a_0^{-r}$ by a path outside of $B(e,r/2)$ in the $a_0a_1$-plane. Therefore, the group $G_m$ is one-ended.

\end{proof}

\subsection{The lower bound}

\subsubsection{Lower bound strategy}
In this section, we prove the lower bound for the group divergence of $G_m$. In order to estimate divergence, we give lower bound estimates for the lengths of open $r$--ball avoidant paths connecting two points on the $r$--sphere. These lower bound estimates are obtained by cutting these paths into pieces using $a_m$--hyperplanes (or $a_m$--corridors) and estimating the lengths of these pieces. We use a $2$--dimensional geometric model $X_m$ for $G_m$ to introduce the hyperplanes and to make the lower bound estimates.

The lower bound for the divergence of $G_m$ is obtained by considering the divergence of the corner $(\alpha,\beta)_e$ based at the identity $e$ in $X_m$, where the geodesic ray $\alpha$ has edges labeled by $a_0$ and the geodesic ray $\beta$ has edges labeled by $a_m$. Proposition~\ref{p3p3} shows that the length of an open $r$--ball avoidant path connecting $\alpha(r)$ to $\beta(r)$ in $X_m$ has length which dominates the function $r^{m-1}f(r)$. 
The main part of the proof of Proposition~\ref{p3p3} is done by induction on the degree of the polynomial portion of the estimate $r^{m-1}f(r)$. The induction step is treated in detail (and in more generality) in Proposition~\ref{p1p1}. In order to prove Proposition~\ref{p1p1} we will need to introduce the notions of a $(0,k)$--ray and of a $k$--corner. These are defined below. Some of these definitions (hyperplane, $k$--corner, and $r$--avoidant path over a $k$--corner) are modeled on the work of Macura~\cite{MR3032700}.

\subsubsection{Lower bound details}
 Let $Y_1$ be the standard presentation $2$--complex for the group $G_1$ with the finite generating set $S_1$. For each $m\geq 2$ we construct a presentation $2$--complex for the group $G_m$ by induction on $m$. Let $C_m$ be the Euclidean unit square with the torus orientation. We label two opposite directed edges by $a_m$ and identify them to obtain a cylinder $U_m$. The remaining two edges of $C_m$ map to loops in $U_m$, and we label them $a_{m-1}$ and $a_0$ respectively. We glue the cylinder $U_m$ to $Y_{m-1}$ by the identification map which is the orientation preserving isometry prescribed by the labeling of the edges. Then the resulting complex $Y_m$ is a graph of spaces with one vertex space $Y_{m-1}$ and one edge space $S_1$. 

Let $X_m=\widetilde{Y}_m$ be the universal cover of $Y_m$. We consider the $1$--skeleton of $X_m$ as the Cayley graph $\Gamma(G_m,S_m)$. For $1\leq i \leq m-1$ the preimage of $Y_i$ in $X_m=\widetilde{Y}_m$ consists of infinitely many disjoint isometrically embedded copies of $X_i=\widetilde{Y}_i$. For each $m\geq 2$ we let $s_m$ be the line segment in $C_m$ connecting the midpoints of the opposite edges labeled $a_m$. We also denote by $s_m$ the image of $s_m$ in $U_m$, as well as its image in $Y_m$ after the gluing. Each component $H$ of the preimage of $s_m\subset Y_m$ in $X_m=\widetilde{Y}_m$ is isometric to the real line and separates $X_m$. Therefore, we call $H$ a \emph{hyperplane} in $X_m$. Moreover, each hyperplane $H$ is contained in a single component $ostar(H)$ of the preimage of $Int(U_m)\subset Y_m$ in $X_m=\widetilde{Y}_m$. We note that $U_m$ is a cylinder. Therefore, the closure $star(H)$ of $ostar(H)$ is isometric to a flat strip.

The following definitions will be used many times in the proof of the lower bound of our group divergence.

\begin{defn}[$k$--rays and $(0,k)$--rays]
Let $m\geq 1$ be an integer. A geodesic ray $\alpha$ in the complex $X_m$ is a \emph{$k$--ray} for $0\leq k\leq m$ if all edges of $\alpha$ are labeled by $a_k$. A geodesic ray $\beta$ in the complex $X_m$ is a \emph{$(0,k)$--ray} for $2\leq k\leq m$ if one of the following holds:
\begin{enumerate}
    \item $\beta$ is a $0$--ray;
    \item $\beta$ is a concatenation of $\sigma_1 \sigma_2$, where $\sigma_1$ is a non-degenerate segment with edges labeled by $a_0$ and $\sigma_2$ is a $k$--ray.
\end{enumerate}
Note that we do not assume that edges of each $k$--ray or edges of each $(0,k)$--ray are oriented away from the base point.


\end{defn}

\begin{rem}
Note that we only define $(0,k)$--ray for $k\geq 2$. The notion of a $(0,1)$--ray does not make sense, because 
$$
|a_0^ra_1^{-r}|_{S_m} \; \leq \; f(r) 
$$
and $f(r)$ may be smaller than $2r$. In particular, the path labeled by $a_0^ra_1^{-r}$ is not a geodesic. Therefore, the definition of $1$--corner given below is more restricted than the definition of $k$--corner for $k \geq 2$. 
\end{rem}

\begin{defn}[$1$--corner]
\label{corner1}
Let $\alpha$ be a $0$--ray and let $\beta$ be a $1$--ray in the complex $X_m$ such that they share the same initial point $x$. Then $(\alpha,\beta)_x$ is called a \emph{$1$--corner} at $x$. 

\end{defn}

We define $k$--corners for $k\geq 2$ as follows.

\begin{defn}[$k$--corner for $k\geq 2$]
\label{corner2}
Let $2\leq k \leq m$ be integers. Let $\alpha$ be a $(0,k)$--ray and let $\beta$ be a $k$--ray in the complex $X_m$ such that they share the same initial point $x$. Then $(\alpha,\beta)_x$ is called a \emph{$k$--corner} at $x$. 

\end{defn}

The following definition recalls Macura's notion of detour paths over corners in \cite{MR3032700}. 

\begin{defn}[An $r$--avoidant path over $k$--corner]
Let $1\leq k\leq m$ be integers and let $(\alpha,\beta)_x$ be a $k$--corner in $X_m$. For each $r>0$ the path outside the open ball $B(x,r)$ in $X_m$ connecting a vertex of $\alpha$ to a vertex of $\beta$ is called an \emph{$r$--avoidant path over the $k$--corner $(\alpha,\beta)_x$} in $X_m$.
\end{defn}

The following lemma and its corollary study some basic facts about $k$--corners for $k\geq 2$ and avoidant paths over them.
\begin{lem}
\label{coolhyp}
Let $k\geq 2$ be an integer. Let $\alpha$ be a $(0,k)$--ray and let $\beta$ be a $k$--ray in the complex $X_k$ such that they share the same initial point $x$. Assume that there is a hyperplane in $X_k$ that is dual to edges labeled by $a_k$ in $\alpha$ and $\beta$. Then $H$ only intersects the first edge of $\beta$.
\end{lem}

\begin{proof}
Let $i$ be the smallest positive integer such that the hyperplane $H_i$ dual to the $i^{th}$ edge of $\beta$ intersects $\alpha$. We observe that if $i \geq 2$, the hyperplane $H_{i-1}$ dual to the $(i-1)^{th}$ edge of $\beta$ must intersect $H_i$ which is a contradiction. Therefore, $i$ must be equal to $1$. Assume that some hyperplane $H_j$ dual to the $j^{th}$ edge of $\beta$ for $j\geq 2$ intersects $\alpha$. Then the hyperplane $H_2$ dual to the second edge of $\beta$ must intersect $\alpha$. We note that for $i=1, 2$ each $star(H_i)-ostar(H_i)$ consists of two bi-infinite geodesics: one has edges labeled by $a_0$ and the other has edges labeled by $a_{k-1}$. Therefore, we have a loop based at the endpoint other than $x$ of the first edge of $\beta$ which is labeled by $a_{k-1}^na_k^pa_0^m$ for $m\neq 0$ and $n\neq 0$. This contradicts to Lemma~\ref{lcool2} and therefore the lemma is proved.  
\end{proof}

\begin{cor}
\label{cocococo}
Let $2\leq k\leq m$ be an integers. Let $\gamma$ be an $r$--avoidant path over a $k$--corner $(\alpha,\beta)_x$ in $X_m$. Then the length of $\gamma$ is at least $r-1$ (therefore, at least $f(r)-1)$.
\end{cor}

\begin{proof}
We assume that $\alpha$ is a $(0,k)$--ray, $\beta$ is a $k$--ray and $\gamma$ connects a vertex $u$ in $\alpha$ to a vertex $v$ in $\beta$. It suffices to prove that $d_{S_m}(u,v)\geq r-1$. By Lemma~\ref{ll1}, we have $d_{S_m}(u,v)=d_{S_k}(u,v)$. Therefore, we only need to show $d_{S_k}(u,v)\geq r-1$ and this inequality is a direct result of Lemma~\ref{coolhyp}.
\end{proof}

The following lemma provides a technique to modify an avoidant path over a $k$--corner to obtain another avoidant path over the same $k$--corner with the length bounded above by the length of the original avoidant path multiplied by a fixed constant. This technique will be used in the proof of Proposition~\ref{p1p1}. 

\begin{lem}
\label{avoidantlem}
Let $m$ be a positive integer. Let $\alpha$ be a geodesic segment labeled by $a_0$ with endpoints $x$ and $y$. Let $z$ be a vertex in the complex $X_m$ and assume that $r=\min\{d_{S_m}(x,z),d_{S_m}(y,z)\}$ is positive. Then we can connect $x$ and $y$ by a path $\beta$ with edges labeled by $a_0$ and $b$ such that $\beta$ lies outside the open ball $B(z,r/2)$ and $\ell(\beta)\leq 11 \ell (\alpha)$.
\end{lem}

\begin{proof}

If $\alpha$ lies outside the open ball $B(z,r/2)$, then we let $\beta=\alpha$. We now assume that $\alpha$ has non-empty intersection with $B(z,r/2)$. Let $u$ be a vertex in $\alpha\cap B(z,r/2)$. We note that the endpoints $x$ and $y$ of $\alpha$ lie outside the open ball $B(z,r)$. Therefore, 
\begin{align*}
    \ell(\alpha)&=d_{S_m}(x,y)=d_{S_m}(x,u)+d_{S_m}(u,y)\\&\geq \bigl(d_{S_m}(x,z)-d_{S_m}(u,z)\bigr)+\bigl(d_{S_m}(y,z)-d_{S_m}(u,z)\bigr)\\&\geq (r-r/2)+(r-r/2)\geq r.
\end{align*}

Since $\alpha$ is a geodesic segment with edges labeled by $a_0$, it is a subsegment of a bi-infinite geodesic $\gamma$ with edges also labeled by $a_0$. By Lemma~\ref{basic} we can choose two vertices $x_1$ and $y_1$ on $\gamma$ such that the following hold:
\begin{enumerate}
    \item The point $x$ lies between two points $u$ and $x_1$ on $\gamma$. The subsegment $\beta_1$ of $\gamma$ connecting $x$ and $x_1$ lies outside the open ball $B(z,r/2)$ and its length is exactly $2r$.
    \item The point $y$ lies between two points $u$ and $y_1$ on $\gamma$. The subsegment $\beta_2$ of $\gamma$ connecting $y$ and $y_1$ lies outside the open ball $B(z,r/2)$ and its length is exactly $2r$.
\end{enumerate}
Since $x$ lies between $u$ and $x_1$ on the bi-infinite geodesic $\gamma$, we have $d_{S_m}(u,x_1)\geq d_{S_m}(x,x_1)=2r$. Therefore, $d_{S_m}(z,x_1)\geq d_{S_m}(u,x_1)-d_{S_m}(u,z)\geq 3r/2$. Similarly, we also have $d_{S_m}(z,y_1)\geq 3r/2$.

Let $x_2=x_1b^r$ and $y_2=y_1b^r$. Let $\beta_3$ be the geodesic connecting $x_1$ and $x_2$ with edges labeled by $b$. Similarly, let $\beta_4$ be the geodesic connecting $y_1$ and $y_2$ with edges labeled by $b$. Then the lengths of both geodesics $\beta_3$ and $\beta_4$ are exactly $r$. Since $d_{S_m}(x_1,z)\geq 3r/2$ and $d_{S_m}(y_1,z)\geq 3r/2$, both geodesics $\beta_3$ and $\beta_4$ lies outside the open ball $B(z,r/2)$.

Since $b$ commutes with $a_0$, we can connect $x_2$ and $y_2$ by a geodesic $\beta_5$ with edges labeled by $a_0$ and $\ell(\beta_5)=d_{S_m}(x_2,y_2)=d_{S_m}(x_1,y_1)$. This implies that $$\ell(\beta_5)=d_{S_m}(x_1,y_1)=d_{S_m}(x_1,x)+d_{S_m}(x,y)+d_{S_m}(y,y_1)=2r+\ell(\alpha)+2r=\ell(\alpha)+4r.$$ We now claim that $\beta_5$ lies outside the open ball $B(z,r/2)$. By the construction, we observe that each point $v$ in $\beta_5$ has distance at least $r$ from $u$. Since $d_{S_m}(u,z)<r/2$, we have $$d_{S_m}(v,z)\geq d_{S_m}(v,u)-d_{S_m}(u,z)\geq r-r/2\geq r/2.$$ In other words, $\beta_5$ lies outside the open ball $B(z,r/2)$ and we proved the claim.

Let $\beta=\beta_1\cup \beta_3\cup \beta_5\cup \beta_4\cup \beta_2$. Then $\beta$ connects two points $x$ and $y$, all edges of $\beta$ are labeled by $a_0$ and $b$, and $\beta$ lies outside the open ball $B(z,r/2)$. We note that the length of $\alpha$ is at least $r$. Therefore,
\begin{align*}
    \ell(\beta)&=\ell(\beta_1)+ \ell(\beta_3)+ \ell(\beta_5)+ \ell(\beta_4)+\ell(\beta_2)\\&=2r+r+\bigl(\ell(\alpha)+4r\bigr)+r+2r\\&= \ell(\alpha)+10r\leq 11 \ell(\alpha).
\end{align*}
\end{proof}

The following proposition is the most technical part of this section. Roughly speaking, the following proposition shows the connection between the length of an avoidant path over a $k$--corner and the length of an avoidant path over a $(k+1)$--corner.

\begin{prop}
\label{p1p1}
Let $d\geq 2$ be an integer and $g\!:(0,\infty)\to(0,\infty)$ be a function. Assume that for all $m\geq d$ all $r$--avoidant paths over an $m$--corner in $X_{m+1}$ have length at least $g(r)$ for $r$ sufficiently large. Then for all $n\geq d+1$ all $r$--avoidant paths over an $n$--corner in the complex $X_{n+1}$ have length at least $(r/180)g(r/4)$ for $r$ sufficiently large.
\end{prop}

\begin{proof}
Let $n\geq d+1$ and let $\gamma$ be an $r$--avoidant path over an $n$--corner $(\alpha,\beta)_x$ in the complex $X_{n+1}$, where $\alpha$ is a $(0,n)$--ray and $\beta$ is an $n$--ray. We can assume $\gamma$ is the $r$--avoidant path with the minimal length over all $n$--corners in the complex $X_{n+1}$. In particular, $\gamma \cap \alpha$ is consists of a vertex $u$ and $\gamma \cap \beta$ is consists of a vertex $v$ in $X_{n}$. We can also assume that $x$ is the identity $e$. Therefore, the vertex $u$ has the form $a_0^sa_{n}^t$ for some integers $s\neq 0$ and $t$ such that $|s|+|t|\geq r$. Similarly, the vertex $v$ has the form $a_n^p$ for some integer $p$ such that $|p|\geq r$. We only prove for the case $p>0$ and the proof for $p<0$ is almost identical.

Since the endpoints of $\gamma$ both lie in $X_{n}\subset X_{n+1}$, the path $\gamma$ is the concatenation $$\sigma_1\tau_1\sigma_2\tau_2\cdots \sigma_{\ell}\tau_{\ell}\sigma_{\ell+1}$$ satisfying the following conditions:
\begin{enumerate}
    \item Each $\sigma_i$ intersects $X_{n}$ only at its endpoints $x_i$ and $y_i$. Here we consider $x_1=u$ and $y_{\ell+1}=v$. We also assume that $x_i$ is also an endpoint of $\tau_{i-1}$ for $i\geq 2$ and $y_i$ is an endpoint of $\tau_{i}$ for $i\leq \ell$; 
    \item Each $\tau_i$ lies completely in the $1$--skeleton of $X_{n}$.
\end{enumerate}

We also observe that $x_i^{-1}y_i$ is a group element in the cyclic subgroup $\langle a_n \rangle$ or in the cyclic subgroup $\langle a_{0} \rangle$ for each $i$. If $x_i^{-1}y_i$ is a group element in the cyclic subgroup $\langle a_n \rangle$, then we replace $\sigma_i$ by a geodesic $\sigma'_i$ labeled by $a_n$. We note that $\sigma'_i$ in this case may have non-empty intersection with the open ball $B(e,r/2)$. If $x_i^{-1}y_i$ is a group element in the cyclic subgroup $\langle a_0 \rangle$, then by Lemma~\ref{avoidantlem} we can replace $\sigma_i$ a path $\sigma'_i$ with edges labeled by $a_0$ and $b$ such that $\sigma'_i$ lies outside the open ball $B(e,r/2)$ and $\ell(\sigma'_i)\leq 11\text{ }d_{S_n}(x_i,y_i)\leq 11 \ell (\sigma_i)$. The new path $\gamma'=\sigma'_1\tau_1\sigma'_2\tau_2\cdots \sigma'_{\ell}\tau_{\ell}\sigma'_{\ell+1}$ lies completely in the $1$--skeleton of $X_n$, shares two endpoints $u$ and $v$ with $\gamma$, and $\ell(\gamma')\leq 11 \ell(\gamma)$. We call each subsegment $\sigma'_i$ of $\gamma'$ labeled by $a_n$ (resp. labeled by $a_0$ and $b$) a short-cut segment of type $1$ (resp. type 2).

We claim that $\gamma'$ does not intersect $\beta$ at any point other than $v$. In fact, assume by the way of contradiction that $\gamma'$ intersects $\beta$ at some point other than $v$. Then some short-cut segment $\sigma'_i$ of type $1$ must contain an edge of $\beta$. Therefore, the endpoint $x_i$ of $\sigma'_i$ has the form $a_n^q$ for some $|q|\geq r$. Therefore, the subpath $\zeta$ of $\gamma$ connecting $u$ and $x_i$ is also an $r$--avoidant path over an $n$--corner in the complex $X_{n+1}$. Also, $|\zeta|<|\gamma|$ which contradicts to the choice of $\gamma$. 

For each positive integer $j$ we call $e_j$ the $j^{th}$ edge of $\beta$. We know that all edges $e_j$ are labeled by $a_n$. Let $H_j$ be the hyperplane of the complex $X_n$ that corresponds to the edge $e_j$. We now consider $r/8\leq j\leq r/4$. Assume that $r\geq 16$. Then $j\geq 2$. Therefore by Lemma~\ref{coolhyp} each hyperplane $H_j$ must intersect $\gamma'$. Let $m_j$ be the point in the intersection $H_j\cap \gamma'$ such that the subpath of $\gamma'$ connecting $m_j$ and $v$ does not intersect $H_j$ at point other than $m_j$. Then $m_j$ is an interior point of an edge $f_j$ labeled by $a_n$ in $\gamma'$. Since $\gamma'$ does not intersect $\beta$ at any point other than $v$, two edges $e_j$ and $f_j$ are distinct. 

Let $\beta_j$ be the path in $star(H_j)-ostar(H_j)$ that connects the terminal $a_n^j$ of the edge $e_j$ to some endpoint $v'_j$ of the edge $f_j$. Then $\beta_j$ is a part of an $(n-1)$--ray. If $f_j$ is not an edge of a short-cut segment of $\gamma'$, then we let $v_j=v'_j$. Otherwise, $f_j$ is an edge in some short-cut segment $\sigma'_{\ell_j}$ of type 1 of $\gamma'$. In this case, we let $v_j$ be the endpoint $y_{\ell_j}$ of $\sigma'_{\ell_j}$. Let $\alpha_j$ be the path in $star(H_{j+1})-ostar(H_{j+1})$ that connects the initial endpoint $a_n^j$ of the edge $e_{j+1}$ to some endpoint $u'_{j+1}$ of the edge $f_{j+1}$. Then $\alpha_j$ is a part of an $0$--ray. If $f_{j+1}$ is not an edge of a short-cut segment of $\gamma'$, then we let $u_{j+1}=u'_{j+1}$. Otherwise, $f_{j+1}$ is an edge in some $\sigma'_{\ell_{j+1}}$. In this case, we let $u_{j+1}$ be the endpoint $x_{\ell_{j+1}}$ of $\sigma'_{\ell_{j+1}}$.

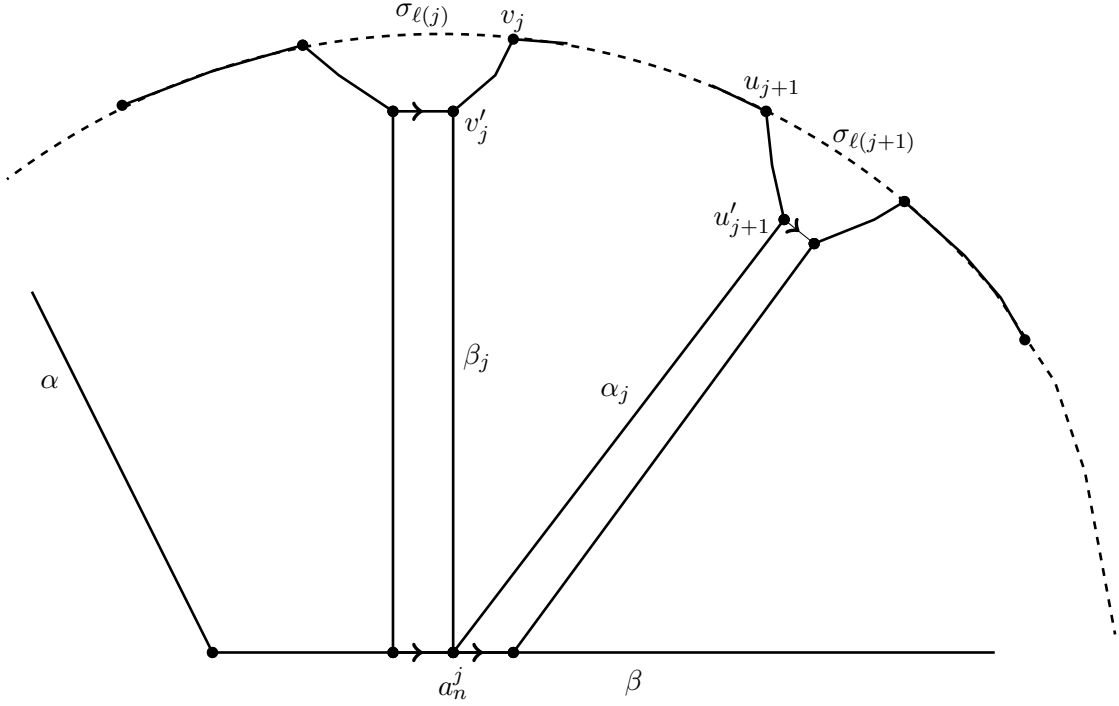
\begin{figure}
  
 \tikzset{->-/.style={decoration={
  markings,
  mark=at position .5 with {\arrow[ultra thick]{>}}},postaction={decorate}}}
   \begin{tikzpicture}[scale=0.8]
   
   \draw[dashed, line width=1.0pt] (11,-2) arc (40:127:12); \draw[dashed, line width=1.0pt] (11,-2) to (12,-3.5) to (12.5,-5) to (13,-7.7);
   
   \draw[line width=1.0pt] (-2,-8) node[circle,fill,inner sep=1.5pt, color=black](1){} to (11,-8);
   \draw[line width=1.0pt] (-2,-8) node[circle,fill,inner sep=1.5pt, color=black](1){} to (-5,-2);
    
    \draw[line width=1.0pt] (2,-8) node[circle,fill,inner sep=1.5pt, color=black](1){} to (1,-8) node[circle,fill,inner sep=1.5pt, color=black](1){}; \draw[->-] (1,-8) node[circle,fill,inner sep=1.5pt, color=black](1){} to (2,-8) node[circle,fill,inner sep=1.5pt, color=black](1){};
    
    \draw[line width=1.0pt] (3,-8) node[circle,fill,inner sep=1.5pt, color=black](1){} to (2,-8) node[circle,fill,inner sep=1.5pt, color=black](1){}; \draw[->-] (2,-8) node[circle,fill,inner sep=1.5pt, color=black](1){} to (3,-8) node[circle,fill,inner sep=1.5pt, color=black](1){};
    
    \draw[line width=1.0pt] (2,-8) node[circle,fill,inner sep=1.5pt, color=black](1){} to (2,1) node[circle,fill,inner sep=1.5pt, color=black](1){};\draw[line width=1.0pt] (1,-8) node[circle,fill,inner sep=1.5pt, color=black](1){} to (1,1) node[circle,fill,inner sep=1.5pt, color=black](1){}; \draw[line width=1.0pt] (1,1) node[circle,fill,inner sep=1.5pt, color=black](1){} to (2,1) node[circle,fill,inner sep=1.5pt, color=black](1){}; \draw[->-] (1,1) node[circle,fill,inner sep=1.5pt, color=black](1){} to (2,1) node[circle,fill,inner sep=1.5pt, color=black](1){};
    
      \draw[line width=1.0pt] (-0.5,2.1) node[circle,fill,inner sep=1.5pt, color=black](1){} to (0.1,1.6) to (1,1);\draw[line width=1.0pt] (2,1) to (2.7,1.6) to  (3,2.2)node[circle,fill,inner sep=1.5pt, color=black](1){};

    \draw[line width=1.0pt] (3,-8) node[circle,fill,inner sep=1.5pt, color=black](1){} to (8,-1.2) node[circle,fill,inner sep=1.5pt, color=black](1){};\draw[line width=1.0pt] (2,-8) node[circle,fill,inner sep=1.5pt, color=black](1){} to (7.5,-.8) node[circle,fill,inner sep=1.5pt, color=black](1){};\draw[->-] (7.5,-0.8) node[circle,fill,inner sep=1.5pt, color=black](1){} -- (8,-1.2) node[circle,fill,inner sep=1.5pt, color=black](1){};
    
     \draw[line width=1.0pt] (9.5,-0.5) node[circle,fill,inner sep=1.5pt, color=black](1){} to (9,-0.8) to (8,-1.2) node[circle,fill,inner sep=1.5pt, color=black](1){};
     
     \draw[line width=1.0pt] (7.2,1) node[circle,fill,inner sep=1.5pt, color=black](1){} to (7.3,0.1) to (7.5,-0.8) node[circle,fill,inner sep=1.5pt, color=black](1){};
     
     \draw[line width=1.0pt] (9.5,-0.5) node[circle,fill,inner sep=1.5pt, color=black](1){} to (10.5,-1.4) to (11.1,-2.1) to (11.5,-2.8) node[circle,fill,inner sep=1.5pt, color=black](1){};
     
     \draw[line width=1.0pt] (7.2,1) node[circle,fill,inner sep=1.5pt, color=black](1){} to (6.3,1.43);
     
     \draw[line width=1.0pt] (3,2.2) to (3.9,2.13);
     
     \draw[line width=1.0pt] (-0.5,2.1) node[circle,fill,inner sep=1.5pt, color=black](1){} to (-2, 1.67) to (-3.5,1.1) node[circle,fill,inner sep=1.5pt, color=black](1){};
    
  \node at (2,-8.5) {$a_n^j$}; \node at (2.4,-3.1) {$\beta_j$}; \node at (4.7,-3.7) {$\alpha_j$}; \node at (6.8,-0.8) {$u'_{j+1}$}; \node at (7.3,1.4) {$u_{j+1}$};  \node at (2.4,0.75) {$v'_j$};\node at (3,2.5) {$v_j$}; \node at (5,-8.5) {$\beta$}; \node at (-4.7,-3.5) {$\alpha$}; \node at (1.5,2.6) {$\sigma_{\ell(j)}$}; \node at (9,0.5) {$\sigma_{\ell(j+1)}$};

      \end{tikzpicture}
\caption{$(\alpha_j,\beta_j)_{a_n^j}$ is a part of a $(n-1)$--corner at $a_n^j$}
\label{Fi1}      
\end{figure}

We see that $(\alpha_j,\beta_j)_{a_n^j}$ is a part of an $(n-1)$--corner at $a_n^j$ (see Figure~\ref{Fi1}). Let $\gamma'_j$ be the subpath of $\gamma'$ that connects $v'_j$ and $u'_{j+1}$. Let $\gamma''_j$ be the subpath of $\gamma'$ that connects $v_j$ and $u_{j+1}$. Let $\gamma_j$ be the subpath of $\gamma$ that connects $v_j$ and $u_{j+1}$. Then by the construction of $\gamma'$ we have $$\ell(\gamma''_j)\leq 11\ell(\gamma_j).$$
Also, $\gamma'_j=\eta_1\cup \gamma''_j\cup \eta_2$, where $\eta_1$ (resp. $\eta_2$) is the (possibly degenerate) subsegment of $\gamma'_j$ connecting $v'_j$ and $v_j$ (resp. $u'_{j+1}$ and $u_{j+1}$). Therefore, $$\ell(\gamma''_j)=\ell(\gamma'_j)-\bigl(\ell(\eta_1)+\ell(\eta_2)\bigr).$$
Since $\eta_1$ and $\eta_2$ are subpaths of short-cut segments of type 1 which are also geodesics, we have $$\ell(\eta_1)=d(v'_j,v_j) \text{ and } \ell(\eta_2)=d(u'_{j+1},u_{j+1}).$$
This implies that $$\ell(\gamma'_j)-\bigl(d(v'_j,v_j)+d(u'_{j+1},u_{j+1})\bigr)\leq 11 \ell(\gamma_j).$$
Therefore, $$d(v'_j,v_j)+\ell(\gamma_j)+d(u'_{j+1},u_{j+1})\geq \ell(\gamma'_j)/11.$$

We note that each short-cut segment $\sigma'$ of type $2$ of $\gamma'$ lies outside the open ball $B(e,r/2)$ by the construction. Therefore, $\sigma'$ also lies outside each open ball $B(a_n^{j},r/4)$ for $r/8 \leq j\leq r/4$. We now consider the case of type 1 short-cut segments $\sigma'$ and we have two cases:

\textbf{Case 1}: For each short-cut segment $\sigma'$ of type $1$ of $\gamma'$ such that $\sigma'\cap \gamma'_j\neq \emptyset$ the intersection $\sigma'\cap \gamma'_j$ lies outside the open ball $B(a_n^{j},r/4)$. Then, $\gamma'_j$ is an $(r/4)$--avoidant path over the $(n-1)$--corner containing $(\alpha_j,\beta_j)_{a_n^j}$ in $X_n$. Also, $n-1\geq d$. Therefore, $\ell(\gamma'_j)\geq g(r/4)$ for $r$ sufficiently large by the induction hypothesis. This implies that $$d(v'_j,v_j)+\ell(\gamma_j)+d(u'_{j+1},u_{j+1})\geq g(r/4)/11.$$

\begin{figure}
  
 \tikzset{->-/.style={decoration={
  markings,
  mark=at position .5 with {\arrow[ultra thick]{>}}},postaction={decorate}}}
   \begin{tikzpicture}[scale=0.9]
   
   \draw[dashed, line width=1.0pt] (11,-2) arc (40:130:12);
   
   \draw[line width=1.0pt] (1,-6) node[circle,fill,inner sep=1.5pt, color=black](1){} to (1,-6); \node at (0.7,-6.4) {$a_n^j$};
  \draw[dashed,line width=1.0pt] (3,-5) arc (0:160:2); 
   \draw[dotted, line width=1.0pt][<->] (0.9,-5.9) -- (-0.7,-4); \node at (-0.1,-5.3) {$r/4$};
 
    \draw[line width=1.0pt] (-0.5,2.1) node[circle,fill,inner sep=1.5pt, color=black](1){} to (1,-5);\draw[line width=1.0pt] (2,-5) arc (0:-180:0.5);\draw[line width=1.0pt] (2,-5) to (3,2.2)node[circle,fill,inner sep=1.5pt, color=black](1){};\draw[->-] (2.1,-4.4) node[circle,fill,inner sep=1.5pt, color=black](1){} -- (2,-5) node[circle,fill,inner sep=1.5pt, color=black](1){};
    
    \draw[line width=1.0pt] (2,-5) node[circle,fill,inner sep=1.5pt, color=black](1){} to[out=-2,in=-100] (8,-1.2) node[circle,fill,inner sep=1.5pt, color=black](1){};\draw[line width=1.0pt] (2.1,-4.4) node[circle,fill,inner sep=1.5pt, color=black](1){} to[out=-2,in=-100] (7.5,-.8) node[circle,fill,inner sep=1.5pt, color=black](1){};\draw[->-] (7.5,-0.8) node[circle,fill,inner sep=1.5pt, color=black](1){} -- (8,-1.2) node[circle,fill,inner sep=1.5pt, color=black](1){};
    
     \draw[line width=1.0pt] (9.5,-0.5) node[circle,fill,inner sep=1.5pt, color=black](1){} to (9,-0.8) to (8,-1.2) node[circle,fill,inner sep=1.5pt, color=black](1){};
     
     \draw[line width=1.0pt] (7.2,1) node[circle,fill,inner sep=1.5pt, color=black](1){} to (7.3,0.1) to (7.5,-0.8) node[circle,fill,inner sep=1.5pt, color=black](1){};
     
     \draw[line width=1.0pt] (9.5,-0.5) node[circle,fill,inner sep=1.5pt, color=black](1){} to (10.5,-1.4) node[circle,fill,inner sep=1.5pt, color=black](1){};
     
     \draw[line width=1.0pt] (7.2,1) node[circle,fill,inner sep=1.5pt, color=black](1){} to (6.3,1.43);
     
     \draw[line width=1.0pt] (3,2.2) to (3.9,2.13);
     
     \draw[line width=1.0pt] (-0.5,2.1) node[circle,fill,inner sep=1.5pt, color=black](1){} to (-2, 1.67) to (-3.5,1.1) node[circle,fill,inner sep=1.5pt, color=black](1){};
    

    
    
   

  \node at (2.7,-2.1) {$a_n^*$}; \node at (4.3,-3.7) {$a_{0}^*$}; \node at (7,0.2) {$a_n^*$}; \node at (1.9,-4.2) {$w$}; \node at (7.1,-0.8) {$w'$}; \node at (7.2,1.4) {$\widetilde{w}$}; \node at (1.8,-4.75) {$f$}; \node at (8,-0.8) {$f'$}; \node at (10.9,-1.2) {$u_{j+1}$}; \node at (-3.5,1.4) {$v_j$}; \node at (3,2.5) {$\widetilde{v}$}; \node at (5.3,2) {$\eta$};

      \end{tikzpicture}
\caption{Some short-cut segment of $\sigma'_j$ intersect the open ball $B(a_n^j,r/4)$ and the subsegment $\eta$ of $\gamma_j$ that connects $\widetilde{v}$ and $\widetilde{w}$ is an $(r/4)$--avoidant path over an $n$--corner in $X_{n+1}$.}
\label{Fi2}      
\end{figure}

\textbf{Case 2}: We now assume that there is a short-cut segment $\sigma'$ of type $1$ of $\gamma'$ such that $\sigma'\cap \gamma'_j\neq \emptyset$ and it intersects the open ball $B(a_n^j,r/4))$ (see Figure~\ref{Fi2}). We will prove that some subsegment of $\gamma_j$ is an $(r/4)$--avoidant path over an $n$--corner in $X_{n+1}$. Let $f$ be an edge of $\sigma'\cap \gamma'_j$ that lies inside the open ball $B(a_n^j,(r/4)+1))$. Then $f$ is labeled by $a_n$. Let $H$ be the hyperplane in $X_n$ that is dual the edge $f$. Then $H$ must intersect $\gamma'_j$. Let $x$ be the point in the intersection $H\cap \gamma'_j$. Then $x$ is an interior point of an edge $f'$ labeled by $a_n$ in $\gamma'_j$. Let $\alpha'$ be the path labeled by $a_0$ in $star(H)-ostar(H)$ that connect a vertex $w$ of $f$ to a vertex $w'$ of $f'$. If $f'$ is not an edge in short-cut segment of $\gamma'$, then $w'$ is a vertex of $\gamma$. In this case, we let $\widetilde{w}=w'$ and $\widetilde{\alpha}=\alpha'$. In the case $f'$ is an edge in short-cut segment of $\sigma''$ of $\gamma'$, we let $\widetilde{w}$ is an endpoint of $\sigma''$ that belongs to $\gamma'_j$, let $\alpha''$ be the subsegment $\sigma''$ connecting $w'$ and $\widetilde{w}$, and let $\widetilde{\alpha}=\alpha'\cup \alpha''$. Therefore, $\widetilde{\alpha}$ is a part of an $(0,n)$--ray. Let $\widetilde{v}$ be the endpoint of $\sigma'$ that belongs to $\gamma_j$ and let $\widetilde{\beta}$ is a subsegment of $\sigma'$ that connects $w$ and $\widetilde{v}$. Then $\widetilde{\beta}$ is a part of an $n$--ray and $(\widetilde{\alpha},\widetilde{\beta})_w$ is a part of an $n$--corner.

Let $\eta$ be the subsegment of $\gamma_j$ that connects $\widetilde{v}$ and $\widetilde{w}$. We note that $\eta$ lies outside the open ball $B(e,r)$ in $X_{n+1}$ and therefore it lies outside the open ball $B(a_n^j,3r/4)$ in $X_{n+1}$. Also, $d(a_n^j,w)<r/4+1\leq r/2$ if we assume that $r>4$. Therefore, $\eta$ lies outside the open ball $B(w,r/4)$. Thus, $\eta$ is an $(r/4)$--avoidant path over an $n$--corner in $X_{n+1}$. Also, $n\geq d+1>d$. Therefore, $$\ell(\gamma_j)\geq \ell(\eta)\geq g(r/4).$$ 

Overall, we always have $$d(v'_j,v_j)+\ell(\gamma_j)+d(u'_{j+1},u_{j+1})\geq g(r/4)/11.$$ Therefore, $$\ell(\gamma)\geq \sum_{r/8\leq j\leq r/4} \bigl(d(v'_j,v_j)+\ell(\gamma_j)+d(u'_{j+1},u_{j+1})\bigr)\geq (\frac{r}{16})\bigl(\frac{g(r/4)}{11}\bigr)\geq \frac{r}{180}g(r/4).$$ 
\end{proof}

In the following, Lemma~\ref{l1l1} and Proposition~\ref{p3p3} provide lower bounds on the lengths of avoidant paths over $k$--corners. We note that Proposition~\ref{p3p3} will be used for the proof of the lower bound of the divergence of the group $G_m$ for $m\geq 3$.

\begin{lem}
\label{l1l1}
Let $k\geq 1$ be an integer. Then all $r$--avoidant paths over an $k$--corner in $X_{k+1}$ have length at least $f(r)-1$ for $r$ sufficiently large.
\end{lem}

\begin{proof}
The proof of the above lemma follows from Corollary~\ref{cocococo} for the case $k\geq 2$ and follows from Lemma~\ref{ll1} and Statement (2) in Lemma~\ref{a1} for the case $k=1$.
\end{proof}



 \begin{prop}
\label{p3p3}
For each integer $d\geq 1$ there is a positive number $n_d$ and a $(d-1)$--degree polynomial $p_d$ with a positive leading coefficient such that the following holds. Let $k\geq d$ be integers and let $(\alpha,\beta)_x$ be a $k$--corner. Let $\gamma$ be an $r$--avoidant path over the $k$--corner $(\alpha,\beta)_x$ in the complex $X_{k}$. Then the length of $\gamma$ is at least $p_d(r)f(r/{n_d})$ for $r$ sufficiently large. 
\end{prop}

\begin{proof}
We first prove the following claim. For each integer $d\geq 1$ there is a positive number $n_d$ and a $(d-1)$--degree polynomial $p_d$ with a positive leading coefficient such that the following holds. Let $k\geq d$ be integers and let $(\alpha,\beta)_x$ be a $k$--corner. Let $\gamma$ be an $r$--avoidant path over the $k$--corner $(\alpha,\beta)_x$ in the complex $X_{k+1}$. Then the length of $\gamma$ is at least $p_d(r)f(r/{n_d})$ for $r$ sufficiently large.

The above claim can be proved by induction on $d$. In fact, the claim is true for the base case $d=1$ due to Lemma~\ref{l1l1} and the fact that $f(r)-1\geq f(r)/2$ for $r$ sufficiently large. Then Proposition~\ref{p1p1} establishes the inductive step and the above claim is proved. By Lemma~\ref{ll1}, we observe that if $\gamma$ is an $r$--avoidant path over a $k$--corner in the complex $X_k\subset X_{k+1}$ then $\gamma$ is also an $r$--avoidant path over the same $k$--corner in the complex $X_{k+1}$. Therefore, the proposition follows from the above claim.
\end{proof}

We now prove the lower bound for the divergence of the groups $G_m$ for $m\geq 3$.

\begin{cor}
\label{lowerkey1}
Let $m\geq 3$ be an integer. Let $\{\delta_\rho\}$ be the divergence of the Cayley graph $\Gamma(G_m,S_m)$. Then $r^{m-1}f(r)\preceq \delta_\rho(r)$ for each $\rho \in (0,1/2]$.
\end{cor}

\begin{proof}
Let $n_m$ be the positive number and $p_m$ be the $(m-1)$--degree polynomial in Proposition~\ref{p3p3}. We will prove that $\delta_\rho (r/\rho)\geq p_{m}(r)f(r/n_m)$ for $r$ sufficiently large. Let $\alpha$ be a $0$--ray and let $\beta$ be a $m$--ray such that they share the initial point at the identity $e$. Then $(\alpha,\beta)_e$ is an $m$--corner. Let $\gamma$ be an arbitrary path which connects $\alpha(r/\rho)$ and $\beta(r/\rho)$ and lies outside the open ball $B(e, r)$. 
Then by Proposition~\ref{p3p3}, the length of the path $\gamma$ is bounded below by $p_{m}(r)f(r/n_m)$ for $r$ sufficiently large. This implies that $\delta_\rho (r/\rho)\geq p_{m}(r)f(r/n_m)$ for $r$ sufficiently large. Therefore, $r^{m-1}f(r)\preceq \delta_\rho(r)$ for each $\rho \in (0,1]$. 
\end{proof}

Before we prove the quadratic lower bound for the divergence of the group $G_2$ we need the following proposition.

\begin{prop}
\label{pcoolcool}
Let $n$ be an arbitrary integer greater than $16$ and $s$ be a generator of $G_2$ in $H-\langle c \rangle $. Let $\gamma$ be a path with endpoints $(a_2s)^{-n}$ and $(a_2s)^{n}$ which avoids the open ball $B(e,n)$. Then the length of $\gamma$ is at least $n^2/16$.
\end{prop}

\begin{proof}
For each $0\leq i \leq n/8$ let $e_i$ be an edge labeled by $a_2$ with endpoints $(a_2s)^{i}$ and $(a_2s)^{i}a_2$. Then the hyperplane $H_i$ of the complex $X_2$ that corresponds to $e_i$ intersects $\gamma$. Let $u_i$ be the point in this intersection such that the subpath of $\gamma$ connecting $u_i$ and $(a_2s)^{n}$ does not intersect $H_i$ at point other than $u_i$. Then $u_i$ is the midpoint of an edge $f_i$ of $\gamma$. 

Let $\alpha_i$ be the path in $star(H_i)-ostar(H_i)$ that connects $(a_2s)^{i}$ to an endpoint of $f_i$. Therefore, $\alpha_i$ is labeled by $a_0$ and the endpoint $v_i$ of $\alpha_i$ in $\gamma$ has the form $(a_2s)^{i}a_0^{m_i}$. Since $|(a_2s)^{i}|_{S_2}\leq 2i\leq n/2$ and $(a_2s)^{i}a_0^{m_i}$ lies outside the open ball $B(e,n)$, then $|m_i|\geq n-n/2\geq n/2$. Let $\beta_i$ be the path in $star(H_i)-ostar(H_i)$ that connects $(a_2s)^{i}a_2$ to an endpoint of $f_i$. Therefore, $\beta_i$ is labeled by $a_1$ and the endpoint $w_i$ of $\beta_i$ in $\gamma$ has the form $(a_2s)^{i}a_2a_1^{n_i}$. Since $|(a_2s)^{i}a_2|_{S_2}\leq 2i+1\leq n/2$ and $(a_2s)^{i}a_2a_1^{n_i}$ lies outside the open ball $B(e,n)$, then $|n_i|\geq n-n/2\geq n/2$. 

For each $1\leq i \leq n/8$ let $\gamma_i$ be the subpath of $\gamma$ that connects $w_{i-1}$ and $v_i$. Therefore, the length of $\gamma_i$ is at least $d_{S_2}(w_{i-1}, v_i)$. Also, $d_{S_2}(w_{i-1}, v_i)=|w_{i-1}^{-1}v_i|_{S_2}=|a_1^{-n_{i-1}}sa_0^{m_i}|_{S_2}$ and the length of the element $a_1^{-n_{i-1}}sa_0^{m_i}$ is $|n_{i-1}|+|m_{i}|+1$. We see this as follows
\begin{eqnarray*}
d_{G_2}(1, a_1^{-n_{i-1}}sa_0^{m_i})  & = & d_{G_1}(1, a_1^{-n_{i-1}}sa_0^{m_i}) \\
& = & d_{H \ast_{\langle c\rangle}{\mathbb Z}^2}(1, a_1^{-n_{i-1}}sa_0^{m_i}) \\
& \geq & d_{{\mathbb Z}^2}(1, a_1^{-n_{i-1}}\langle c \rangle) + d_H(\langle c\rangle, s\langle c \rangle) + d_{{\mathbb Z}^2}(\langle c \rangle, a_0^{m_i})\\
& = & |n_{i-1}| + 1  + |m_i|.
\end{eqnarray*}
The first two equalities hold because the subgroup inclusions are isometric embeddings with the respective generating sets. The inequality holds from Bass-Serre theory (of free products with amalgamation). For the last equality, we have $d_H(\langle c\rangle, s\langle c \rangle) = 1$ because $s \not\in \langle c\rangle$. The remaining parts are easily seen by killing $c=a_0a_1^{-1}$ in ${\mathbb Z}^2$ to get ${\mathbb Z}$ generated by $a_0=a_1$.

Therefore, $$\ell(\gamma_i)\geq |n_{i-1}|+|m_{i}|+1\geq n/2+n/2+1\geq n.$$
This implies that $$\ell(\gamma)\geq \sum_{1\leq i\leq n/8} \ell(\gamma_i)\geq n^2/16.$$
\end{proof}

We now prove the quadratic lower bound for the divergence of the group $G_2$.

\begin{cor}
\label{lowerkey2}
Let $\{\delta_\rho\}$ be the divergence of the Cayley graph $\Gamma(G_2,S_2)$. Then $r^2\preceq \delta_\rho(r)$ for each $\rho \in (0,1/2]$.
\end{cor}

\begin{proof}
We will prove that $\delta_\rho (r/\rho)\geq r^2/256-2r/\rho$ for $r$ sufficiently large. Let $s$ be a generator of $G_2$ in $H-\langle c \rangle $. Let $\alpha$ be bi-infinite geodesic containing the identity element $e$ with edges labeled by $a_2$ and $s$ alternately. Let $x$ and $y$ be the two points in the intersection $\alpha \cap S(e,r/\rho$). We assume that the subsegment of $\alpha$ from $e$ to $x$ traces each edge of $\alpha$ in the positive direction and the subsegment of $\alpha$ from $e$ to $y$ traces each edge of $\alpha$ in the negative direction. Let $\beta$ be an arbitrary path with endpoints $x$ and $y$ that lies outside the ball $B(e,r)$. 
Let $n$ be the largest integer such that $n\leq r/2$. Therefore, $n\geq r/2-1\geq r/4$ for $r$ sufficiently large. Let $\alpha_1$ be a subsegment of $\alpha$ that connects $(a_2s)^n$ to $x$. Therefore, $\alpha_1$ lies outside the open ball $B(e,n)$ and has the length bounded above by $r/\rho$. Similarly, let $\alpha_2$ be a subsegment of $\alpha$ that connects $(a_2s)^{-n}$ to $y$. Therefore, $\alpha_2$ lies outside the open ball $B(e,n)$ and has the length bounded above by $r/\rho$. Let $\gamma=\alpha_1 \cup \beta \cup \alpha_2$. Then, $\gamma$ is a path with endpoints $(a_2s)^{-n}$ and $(a_2s)^{n}$ which avoids the open ball $B(e,n)$. Therefore, the length of $\gamma$ is at least $n^2/16$ by Proposition~\ref{pcoolcool}. Therefore, $$\ell(\beta)\geq \ell(\gamma)-2r/\rho\geq n^2/16-2r/\rho\geq r^2/256-2r/\rho.$$ Thus, $\delta_\rho (r/\rho)\geq r^2/256-2r/\rho$ for $r$ sufficiently large. This implies that $r^2\preceq \delta_\rho(r)$.  
\end{proof}



\section{Geodesic divergence}
Proposition~\ref{cyclicdivergence} establishes the geodesic divergence statements in the Main Theorem; namely, that ${\rm Div}^{G_m}_{\langle a_m\rangle}$ is equivalent to $r^{m-1}f(r)$. In Proposition~\ref{lem_CS} we prove that the divergence of a contracting quasi-geodesic is at least quadratic. The latter result can be found implicitly in the literature (for example, it can be deduced by combining techniques of Lemma 6.5 of \cite{MR3361149} and Proposition 3.5 of \cite{RST2018}), but we provide a detailed proof here for completeness.

\begin{prop}
\label{cyclicdivergence}
Let $m\geq 2$ be an integer. Let $\alpha_m$ be a bi-infinite geodesic in the Cayley graph $\Gamma(G_m,S_m)$ with edges labeled by $a_m$. Then the divergence of $\alpha_m$ is equivalent to the function $r^{m-1}f(r)$. 
\end{prop}

\begin{proof}
Without loss of generality we can assume that $\alpha_m(0)=e$ and $\alpha_m(1)=a_m$. Let $\beta:[0,\infty)\to \Gamma(G_m,S_m)$ be a $0$--ray with $\beta(0)=e$. By Lemma~\ref{ag} there is a number $M>0$ such that the following hold. Let $r>0$ be an arbitrary number. There are a path $\gamma_1$ outside the open ball $B(\alpha_m(0),r)$ connecting $\alpha_m(-r)$ and $\beta(r)$ and a path $\gamma_2$ outside the open ball $B(\alpha_m(0),r)$ connecting $\alpha_m(r)$ and $\beta(r)$ such that the lengths of $\gamma_1$ and $\gamma_2$ are both bounded above by $Mr^{m-1} \bigl(f(Mr)+1\bigr)$. Therefore, the path $\gamma=\gamma_1\cup\gamma_2$ lies outside the open ball $B(\alpha_m(0),r)$, connects $\alpha_m(-r)$ and $\alpha_m(r)$, and has length at most $2Mr^{m-1} \bigl(f(Mr)+1\bigr)$. This implies that ${\rm Div}_{\alpha_m}(r)\leq 2Mr^{m-1} \bigl(f(Mr)+1\bigr)$.

We now prove a lower bound for ${\rm Div}_{\alpha_m}$. Let $n_m$ be the constant and $p_m$ be the $(m-1)$--degree polynomial in Proposition~\ref{p3p3}. Let $\gamma'$ be an arbitrary path outside the open ball $B(\alpha_m(0),r)$ connecting $\alpha_m(-r)$ and $\alpha_m(r)$. Let $e_1$ be the edge of $\alpha_m$ with endpoints $e$ and $a_m$. Then the hyperplane $H$ of the complex $X_m$ corresponding to $e_1$ must intersect $\gamma'$. Therefore, there is a $0$--ray $\beta_1$ with initial point at $e$ that intersects $\gamma'$ at some vertex $v$. This implies that the subpath $\gamma_1$ of $\gamma'$ connecting $\alpha_m(r)$ and $v$ is an $r$--avoidant path over the $m$--corner $({\alpha_m}_{|[0,\infty)},\beta_1)_e$. By Proposition~\ref{p3p3} for $r$ sufficiently large we have $$\ell(\gamma')\geq \ell(\gamma_1)\geq p_m(r)f(r/n_m).$$ This implies that ${\rm Div}_{\alpha_m}(r)\geq p_m(r)f(r/n_m)$ for $r$ sufficiently large. Therefore, the divergence of $\alpha_m$ is equivalent to the function $r^{m-1}f(r)$.
\end{proof}

We now give the proof for the fact the divergence of a contracting quasi-geodesic is at least quadratic. First we need the following lemma.

\begin{lem}
\label{bbaa}
Let $\alpha\!: (-\infty,\infty)\to X$ be an $(L,C)$--quasi-geodesic in a geodesic space $X$ and let $D\geq 1$ be a constant. For all $r > 4L^3(C+1)+4LC$ and all paths $\gamma$ lying outside the open ball $B(\alpha(0),r/L -C)$ and connecting $\alpha(-r)$ and $\alpha(r)$, there exist points $x$ and $y$ in $\gamma$ with the following properties:
\begin{enumerate}
\item $d(x,\alpha)=d(y,\alpha)=s$ and $d(x,y)\geq 6Ds$, and 
\item the subsegment of $\gamma$ connecting $x$ and $y$ lies outside the open $s$--neighborhood of $\alpha$,
\end{enumerate}
where $s= r/[(8L^3+4L)(6D+4)]$. 
\end{lem}

\begin{proof}
We first claim that $\gamma$ does not lie in the $(2s)$--neighborhood of $\alpha$. We assume by the way of contradiction that $\gamma$ lies in the $(2s)$--neighborhood of $\alpha$. Let $\alpha(-r)=x_0,x_1,\cdots,x_n=\alpha(r)$ be points $\gamma$ such that $d(x_{i-1},x_i)<1$ for each $i\in\{1,2,\cdots,n\}$. For each $i\in \{1,2,\cdots,n-1\}$ we let $t_i$ in $(-\infty,\infty)$ such that $d(x_i,\alpha(t_i))<2s$. We also let $t_0=-r$ and $t_n=r$. For each $i\in \{1,2,\cdots,n\}$ we let $I_i$ be the interval in $(-\infty,\infty)$ with endpoints $t_{i-1}$ and $t_i$. Then we observe that $[-r,r]\subset \bigcup I_i$. Therefore, $0\in I_p$ for some $p\in \{1,2,\cdots,n\}$.

We remind the reader that the endpoints of $I_p$ are $t_{p-1}$ and $t_p$. By the choice of $t_{p-1}$ and $t_p$ and the triangle inequality, we have 
$$d(\alpha(t_{p-1}),\alpha(t_p))\leq d(\alpha(t_{p-1}),x_{p-1})+d(x_{p-1},x_p)+d(x_p,\alpha(t_p))<2s+1+2s\leq 4s+1.$$
Since $\alpha$ is an $(L,C)$--quasi-geodesic, we have $|t_p-t_{p-1}|< L(4s+1+C)$. Also, $0$ lies between $t_{p-1}$ and $t_p$. Therefore, $|t_p-0|<L(4s+1+C)$. This implies that
\begin{align*}
  d(x_p,\alpha(0))&\leq d(x_p,\alpha(t_p))+d(\alpha(t_p),\alpha(0))\\&< 2s+L|t_p-0|+C\\&\leq 2s+L^2(4s+1+C)+C\\&\leq (4L^2+2)s+L^2(C+1)+C\\&\leq \frac{r}{2L}+(\frac{r}{2L}-C)\leq \frac{r}{L}-C 
\end{align*}
that contradicts to the fact $\gamma$ lies out side the open ball $B(\alpha(0),r/L -C)$. Therefore, $\gamma$ does not lie in the $(2s)$--neighborhood of $\alpha$.

Now we let $\alpha(-r)=y_0,y_1,\cdots,y_m=\alpha(r)$ be points $\gamma$ such that the following hold:
\begin{enumerate}
    \item $d(y_i,\alpha)=s$ for each $i\in \{1,2,\cdots, m-1\}$;
    \item For each $i\in \{1,2,\cdots, m\}$ either the subpath $\gamma_i$ of $\gamma$ connecting $y_{i-1}$ and $y_i$ lies completely outside the open $s$--neighborhood of $\alpha$ or all the points in $\gamma_i$ excepts its endpoints $y_{i-1}$ and $y_i$ lies inside the open $s$--neighborhood of $\alpha$. 
\end{enumerate}
We let $s_i$ in $(-\infty,\infty)$ such that $d(y_i,\alpha(s_i))<2s$ for $i\in \{1,2,\cdots,m-1\}$. We let $s_0=-r$ and $s_m=r$. For each $i\in \{1,2,\cdots,m\}$ we let $J_i$ be the interval with endpoints $s_{i-1}$ and $s_i$. Then we observe that $[-r,r]\subset \bigcup J_i$. Then $0\in J_q$ for some $q\in \{1,2,\cdots,m\}$.

We remind the reader that the endpoints of $J_p$ are $s_{q-1}$ and $s_q$. If the all points of the subsegment $\gamma_q$ of $\gamma$ excepts $y_{q-1}$ and $y_q$ lie in open $s$--neighborhood of $\alpha$, then we can obtain a contradiction by using a similar argument as above. Therefore, the subsegment $\gamma_q$ of $\gamma$ must lie outside the open $s$--neighborhood of $\alpha$.

We observe that 
\begin{align*}
    d(\alpha(s_q),\alpha(0))&\geq d(y_q,\alpha(0))-d(y_q,\alpha(s_q))\\&\geq (\frac{r}{L}-C)-2s\\&\geq \frac{r}{L}-\frac{r}{4L}-\frac{r}{4L}\geq \frac{r}{2L}.
\end{align*}
Therefore, $$|s_q-0|\geq \frac{r}{2L^2}-\frac{C}{L}\geq \frac{r}{4L^2}.$$
Similarly, $|s_{q-1}-0|\geq r/(4L^2)$. Since $0$ lies between $s_{q-1}$ and $s_q$, we have $|s_q-s_{q-1}|\geq r/(2L^2)$. Therefore,

\begin{align*}
   d(y_p,y_{p-1})&\geq d(\alpha(s_p),\alpha(s_{p-1}))-d(\alpha(s_p), y_p)-d(\alpha(s_{p-1}),y_{p-1})\\&\geq \biggl(\frac{1}{L}|s_q-s_{q-1}|-C\biggr)-2s-2s\\&\geq\frac{r}{2L^3}-C-4s\\&\geq \frac{r}{4L^3}-4s\\&\geq (6D+4)s-4s\geq 6Ds. 
\end{align*}
This implies that $x=y_{q-1}$ and $y=y_q$ are the desired points on $\gamma$.
\end{proof}

We now prove the fact the divergence of a contracting quasi-geodesic is at least quadratic in the following proposition.

\begin{prop}
\label{lem_CS}
Let $\alpha\!: (-\infty,\infty)\to X$ be an $(L,C)$--quasi-geodesic in a geodesic space $X$. Assume that $\alpha$ is also an $(A,D)$--contracting quasi-geodesic. Then the divergence of $\alpha$ is at least quadratic.
\end{prop}

\begin{proof}
Since $\alpha$ is an $(A,D)$--contracting quasi-geodesic, there exist a map $\pi_\alpha\!:X\to \alpha$ satisfying:
\begin{enumerate}
\item $\pi_\alpha$ is $(D,D)$--coarsely Lipschitz; 
\item For any $y\in \alpha$, $d\bigl(y,\pi_\alpha(y)\bigr)\leq D$;
\item For all $x\in X$, if we set $R=A d(x,\alpha)$, then $\diam\bigl(\pi_\alpha\bigl(B_R(x)\bigr)\bigr)\leq D$.
\end{enumerate}

We first show that for all $x \in X$, \[d(x,\pi_\alpha(x)) \leq 2D d(x,\alpha) + 4D.\]
Let $y \in \alpha$ such that $d(x,y) \leq d(x,\alpha)+1$. Then from the definition of $(A,D)$--contracting we have
\begin{align*}
    d(x,\pi_\alpha(x)) \leq& d(x,y) + d\bigl(y,\pi_\alpha(y)\bigr) + d\bigl(\pi_\alpha(y),\pi_\alpha(x)\bigl)\\
    \leq& d(x,\alpha)+1 + D + D  d(x,y)+ D\\
    \leq& d(x,\alpha)+1 + D + D(d(x,\alpha)+1)+ D\\
    \leq& (D+1)d(x,\alpha) +3D +1\\
    \leq& 2D d(x,\alpha) + 4D.
\end{align*}

We now prove that $$\rm Div_\alpha(r)\geq \biggl(\frac{A}{4(8L^3+4L)^2(6D+4)^2}\biggr)r^2$$
for each $r > 4L^3(C+1)+4LC+8(8L^3+4L)(6D+4)$. Let $s= r/[(8L^3+4L)(6D+4)]$ and let $\gamma$ be a path outside the open ball $B(\alpha(0),r/L -C)$. Then by Lemma~\ref{bbaa} we can find two points $x$ and $y$ in $\gamma$ with the following properties:
\begin{enumerate}
\item $d(x,\alpha)=d(y,\alpha)=s$ and $d(x,y)\geq 6 Ds$;
\item The subsegment $\eta$ of $\gamma$ connecting $x$ and $y$ lies outside the open $s$--neighborhood of $\alpha$.
\end{enumerate}
Therefore, 
\begin{align*}
d\bigl(\pi_\alpha(x),\pi_\alpha(y)\bigr) &\geq d(x,y)-d\bigl(x,\pi_\alpha(x)\bigr)-d\bigl(y,\pi_\alpha(y))\bigr)\\
&\geq 6Ds - (2Dd(x,\alpha)+4D)-(2Dd(y,\alpha)+4D) \\
&\geq 6Ds-2(2Ds+4D)\\&\geq 2Ds-8D\geq Ds.
\end{align*}
Let $R=As$, let $x=x_0,x_1, x_2,\cdots,x_n=y$ be points in $\eta$, and let $\eta_i$ be the subsegment of $\eta$ connecting $x_{i-1}$ and $x_i$ for $i\in \{1,2\cdots,x_n\}$ such that $R/4\leq \ell(\eta_i)\leq R/2$ and $\ell(\eta)=\sum\limits_{i=1}^n \ell(\eta_i)$. This implies \[\ell(\eta)=\sum\limits_{i=1}^n \ell(\eta_i)\geq \frac{nR}{4}.\]

Since $\pi_\alpha$ is an $(A,D)$--contracting map and $d(x_{i-1}, x_i)<Ad(x_{i-1}, \alpha)$, we have $d\bigl(\pi_\alpha(x_{i-1}),\pi_\alpha(x_i)\bigr)\leq D$ for each $1\leq i \leq n$. Thus \[
    d\bigl(\pi_\alpha(x),\pi_\alpha(y)\bigr)\leq \sum\limits_{i=1}^n d\bigl(\pi_\alpha(x_{i-1}),\pi_\alpha(x_i)\bigr)\leq nD.
\]
Since $d\bigl(\pi_\alpha(x),\pi_\alpha(y)\bigr) \geq Ds$, we have $n\geq s$. Therefore, $$\ell(\eta)\geq \frac{nR}{4}\geq \frac{sR}{4}\geq \frac{As^2}{4}\geq \frac{Ar^2}{4(8L^3+4L)^2(6D+4)^2}.$$
This implies that $$\rm Div_\alpha(r)\geq \biggl(\frac{A}{4(8L^3+4L)^2(6D+4)^2}\biggr)r^2$$
for each $r > 4L^3(C+1)+4LC+8(8L^3+4L)(6D+4)$.
Thus, the divergence of $\alpha$ is at least quadratic. 
\end{proof}

\bibliographystyle{alpha}
\bibliography{Tran}
\end{document}